\documentclass[reqno]{amsart}

\usepackage{amsmath,amssymb,amsthm,amsfonts,enumerate, enumitem,hyperref,cleveref}
\usepackage{dsfont}
\usepackage[all]{xy}
\usepackage[T1]{fontenc}
\usepackage{tikz}
\usepackage{pgfplots}
\pgfplotsset{compat=1.15}
\usepackage{mathrsfs}
\usetikzlibrary{arrows}

\DeclareMathOperator{\Max}{Max}
\DeclareMathOperator{\Min}{Min}

\theoremstyle{plain}
\newtheorem{theorem}{Theorem}[section]
\newtheorem{corollary}[theorem]{Corollary}
\newtheorem{proposition}[theorem]{Proposition}
\newtheorem{lemma}[theorem]{Lemma}

\newtheorem*{conjecture}{Conjecture}

\theoremstyle{definition}
\newtheorem{definition}[theorem]{Definition}
\newtheorem{remark}[theorem]{Remark}
\newtheorem{example}[theorem]{Example}

\crefrangeformat{equation}{#3(#1)#4--#5(#2)#6}

\crefname{theorem}{Theorem}{Theorems}
\crefname{lemma}{Lemma}{Lemmas}
\crefname{corollary}{Corollary}{Corollaries}
\crefname{proposition}{Proposition}{Propositions}
\crefname{definition}{Definition}{Definitions}
\crefname{example}{Example}{Examples}
\crefname{remark}{Remark}{Remarks}
\crefname{conjecture}{Conjecture}{Conjectures}
\crefname{section}{Section}{Sections}
\crefname{equation}{\unskip}{\unskip}
\crefname{enumi}{\unskip}{\unskip}
\crefname{subsection}{Subsection}{Subsections}

\newcommand{\0}{\theta}

\newcommand{\G}{\Gamma}
\newcommand{\af}{\alpha}

\newcommand{\vf}{\varphi}

\newcommand{\dl}{\delta}

\newcommand{\sg}{\sigma}

\newcommand{\B}{\mathcal{B}}

\newcommand{\m}{{}^{-1}}
\newcommand{\sst}{\subseteq}
\newcommand{\impl}{\Rightarrow}
\newcommand{\gen}[1]{\langle #1\rangle}
\newcommand{\lf}{\lfloor}
\newcommand{\rf}{\rfloor}

\newcommand{\id}{\mathrm{id}}
\renewcommand{\iff}{\Leftrightarrow}

\begin{document}
	\title[Commutativity preservers of incidence algebras]{Commutativity preservers of incidence algebras}	
	\author{\'Erica Z. Fornaroli}
	\address{Departamento de Matem\'atica, Universidade Estadual de Maring\'a, Maring\'a, PR, CEP: 87020--900, Brazil}
	\email{ezancanella@uem.br}
	
	\author{Mykola Khrypchenko}
	\address{Departamento de Matem\'atica, Universidade Federal de Santa Catarina,  Campus Reitor Jo\~ao David Ferreira Lima, Florian\'opolis, SC, CEP: 88040--900, Brazil}
	\email{nskhripchenko@gmail.com}
	
	\author{Ednei A. Santulo Jr.}
	\address{Departamento de Matem\'atica, Universidade Estadual de Maring\'a, Maring\'a, PR, CEP: 87020--900, Brazil}
	\email{easjunior@uem.br}
	
	\subjclass[2020]{Primary: 16S50, 15A86, 17B60, 17B40; secondary: 06A07, 05C38}
	\keywords{Commutativity preserver, incidence algebra, maximal chain, diagonal element}
	
	\begin{abstract}
		Let $I(X,K)$ be the incidence algebra of a finite connected poset $X$ over a field $K$ and $D(X,K)$ its subalgebra consisting of diagonal elements. We describe the bijective linear maps $\vf:I(X,K)\to I(X,K)$ that strongly preserve the commutativity and satisfy $\vf(D(X,K))=D(X,K)$. We prove that such a map $\vf$ is a composition of a commutativity preserver of shift type and a commutativity preserver associated to a quadruple $(\0,\sg,c,\kappa)$ of simpler maps $\0$, $\sg$, $c$ and a sequence $\kappa$ of elements of $K$.
	\end{abstract}
	
	\maketitle
	
	\tableofcontents
	
	\section*{Introduction}
	
	Let $K$ be a field. A $K$-linear map $\vf:A\to B$ between two $K$-algebras {\it preserves the commutativity} (or is a {\it commutativity preserver}) if for any $a,b\in A$
	\begin{align}\label{commpreservdef}
		[a,b]=0 \impl [\vf(a),\vf(b)]=0.
	\end{align}
	Moreover, we say that a commutativity preserver $\vf$ is {\it strong}, whenever the converse implication of \cref{commpreservdef} also holds. In particular, every Lie homomorphism (isomorphism) is a (strong) commutativity preserver. We draw the reader's attention to the fact that the notion of a strong commutativity preserver used here is different from that considered in \cite{Chen_Liu, Chen_Zhao}. Observe that our notion of a strong commutativity preserver is the same as a {\it map preserving commutativity in both directions} in \cite{Bresar2007,Marcoux-Sourour99}. 
	
	Commutativity preservers, such as Lie homomorphisms and automorphisms, have been studied in the context of matrix algebras and their generalizations along the last decades (e.g.~\cite{Benkovic-Eremita04, Bresar-Semrl06, Bre93, Choi-Jafarian, Marcoux-Sourour99, Omladic-Radjavi-Semrl01, Semrl08, Slowik16, Watkins76}). More specifically, Watkins first obtained in \cite{Watkins76} a characterization of the injective maps which preserve commutativity for the full matrix algebras $M_n(K)$ over an algebraically closed field of characteristic zero. Marcoux and Sourour described in \cite{Marcoux-Sourour99} the strong commutativity preservers of the upper triangular matrix algebras $T_n(K)$ and, as aftermath, also obtained a description for the Lie automorphisms of such algebras. Bre\v{s}ar and \v{S}emrl gave in \cite{Bresar-Semrl06} a characterization of commutativity preservers on the full matrix algebras $M_n(K)$ over a field, for $n>2$. In most of the references above, strong commutativity preservers are sums of central-valued maps and scalar multiples of either automorphisms or anti-automorphisms. Such maps are particular cases of strong commutativity preservers of Lie type, as in \cref{Lie-type-def}. The other articles cited above deal with specific classes of commutativity preservers on matrix algebras and their generalizations. 
	
	In this paper we characterize a class of bijective strong commutativity preservers on an incidence algebra $I(X,K)$ of a finite partially ordered set $X$. The particular case of Lie automorphisms of $I(X,K)$ was studied in~\cite{FKS,FKS2}. Recall that $I(X,K)$ is isomorphic to a subalgebra of $T_n(K)$ formed by matrices having zeros in some fixed positions. In particular, it is isomorphic to $T_n(K)$ when $X$ is a chain of cardinality $n$. However, our approach is different from that commonly applied in the case of $T_n(K)$.
	
	The paper is organized as follows. In \cref{sec-prelim} more general definitions and results to be used in the next sections are presented. In \cref{sec-comm-preservers} we study the behavior of a commutativity preserver on the standard basis of an incidence algebra and introduce commutativity preservers {\it of shift type}, which will appear in our final description. In \cref{sec-diag-pres} we study bijective strong commutativity preservers that also preserve diagonality, i.e., the ones mapping diagonal elements to diagonal elements, and obtain our main results. Namely, we show in \cref{vf-decomp-S_af-tau} that every such map is the composition of a commutativity preserver of shift type and a commutativity preserver mapping the Jacobson radical of the incidence algebra to itself (called a {\it pure commutativity preserver}). The latter are characterized in terms of a triple $(\theta,\sigma,c)$ of simpler maps in \cref{vf-decomp-as-tau_0_sg_nu_s} resulting in a full characterization of all bijective strong commutativity preservers which preserve diagonality. We end the paper with some applications, examples and a conjecture.
	
	\section{Preliminaries}\label{sec-prelim}
	
	\subsection{Centralizers and commutativity preservers}
	
	Given an associative unital algebra $(A,\cdot)$, we use the standard notation $[a,b]$ for the Lie product $a\cdot b-b\cdot a$. Let $W\sst Y\sst A$. The {\it centralizer} of $W$ in $Y$ is $C_Y(W):=\{y\in Y: [w,y]=0,\forall w\in W\}$. If $Y$ is a subspace (resp.~subalgebra) of $A$, then $C_Y(W)$ is a subspace (resp.~subalgebra) of $A$. We write $Z(W)$ for $C_W(W)$ and call it the {\it center} of $W$.
	
	We naturally extended the notions of a commutativity preserver and a strong commutativity preserver to maps between subsets of some algebras. If $\vf:A\to B$ preserves the commutativity, then
	\begin{align}\label{vf(C_A(X))-sst-C_B(vf(X))}
		\vf(C_A(W))\sst C_B(\vf(W))    
	\end{align}
	for any $W\sst A$. When $\vf$ is surjective and strong, the inclusion \cref{vf(C_A(X))-sst-C_B(vf(X))} becomes equality. 
	
	\begin{remark}\label{ker-vf-strong}
		Let $\vf:A\to B$ be a strong commutativity preserver. Then $\ker\vf\sst Z(A)$. Moreover, if $A$ is central, then $\vf$ is injective $\iff\vf(1)\ne 0$. 
		
		Indeed, assume that $\vf(a)=0$. Then $[\vf(a),\vf(b)]=0$ for all $b\in A$, so $[a,b]=0$ for all $b\in A$, i.e. $a\in Z(A)$. If $A$ is central, then for any $0\ne a\in Z(A)$ we have $a=k\cdot 1$, where $k\in K^*$, so $\vf(1)\ne 0$ implies $\vf(a)=k\vf(1)\ne 0$.
	\end{remark}

	\subsection{Posets and incidence algebras}
	
	Let $(X,\le)$ be a partially ordered set (which we usually shorten to ``poset'') and $x,y\in X$. The \emph{interval} from $x$ to $y$ is the set $\lfloor x,y \rfloor=\{z\in X : x\leq z\leq y\}$. The poset $X$ is said to be \textit{locally finite} if all the intervals of $X$ are finite. A \textit{chain} in $X$ is a linearly ordered (under the induced order) subset of $X$. The \textit{length} of a finite chain $C\sst X$ is defined to be $|C|-1$. The \textit{length} of a finite poset $X$, denoted by $l(X)$, is the maximum length of all chains $C\sst X$. For simplicity, we write $l(x,y)$ for $l(\lf x,y\rf)$. A {\it walk} in $X$ is a sequence $x_0,x_1,\dots,x_m\in X$, such that $x_i$ and $x_{i+1}$ are comparable and $l( x_i,x_{i+1})=1$ (if $x_i\le x_{i+1}$) or $l(x_{i+1},x_i)=1$ (if $x_{i+1}\le x_i$) for all $i=0,\dots,m-1$. A walk $x_0,x_1,\dots,x_m$ is {\it closed} if $x_0=x_m$. We say that $X$ is {\it connected} if for any pair of $x,y\in X$ there is a walk $x=x_0,\dots,x_m=y$. We will denote by $\Min(X)$ (resp.~$\Max(X)$) the set of minimal (resp.~maximal) elements of $X$ and by $X^2_<$ the set of pairs $(x,y)\in X^2$ with $x<y$.
	
	Let $X$ be a locally finite poset and $K$ a field. The \emph{incidence algebra} $I(X,K)$ of $X$ over $K$ (see~\cite{Rota64}) is the $K$-space of functions $f:X\times X\to K$ such that $f(x,y)=0$ if $x\nleq y$. This is an associative $K$-algebra endowed with the multiplication given by the convolution
	$$
	(fg)(x,y)=\sum_{x\leq z\leq y}f(x,z)g(z,y),
	$$
	for any $f, g\in I(X,K)$. The algebra $I(X,K)$ is unital and its identity element $\delta$ is given by
	\begin{align*}
		\delta(x,y)=
		\begin{cases}
			1, & x=y,\\
			0, & x\ne y.
		\end{cases}
	\end{align*}
	
	Throughout the rest of the paper $X$ will stand for a connected finite poset with $|X|>1$. Then $I(X,K)$ admits the standard basis $\{e_{xy} : x\leq y\}$, where
	\begin{align*}	
		e_{xy}(u,v)=
		\begin{cases}
			1, & (u,v)=(x,y),\\
			0, & (u,v)\ne(x,y).
		\end{cases}
	\end{align*}
	We will write $e_{x}=e_{xx}$. More generally, given $Y\sst X$, we introduce $e_Y:=\sum_{y\in Y}e_y$. It is a well-known fact (see~\cite[Theorem~4.2.5]{SpDo}) that the Jacobson radical $J(I(X,K))$ of $I(X,K)$ is the subspace with basis $\B:=\{e_{xy} : x<y\}$. {\it Diagonal elements} of $I(X,K)$ are those $f\in I(X,K)$ satisfying  $f(x,y)=0$ for $x\neq y$. They form a commutative subalgebra $D(X,K)$ of $I(X,K)$ spanned by $\{e_{x} : x \in X\}$. Clearly, each $f\in I(X,K)$ can be uniquely written as $f=f_D+f_J$ with $f_D\in D(X,K)$ and $f_J\in J(I(X,K))$.
	
	We recall some definitions from \cite{FKS}. Let $C: u_1<u_2<\dots<u_m$ be a maximal chain in $X$. A bijection $\0:\B\to \B$ is \textit{increasing (resp.~decreasing) on $C$} if there exists a maximal chain $D: v_1<v_2<\dots<v_m$ in $X$ such that $\0(e_{u_iu_j})=e_{v_iv_j}$ for all $1\le i<j\le m$ (resp.~$\0(e_{u_iu_j})=e_{v_{m-j+1}v_{m-i+1}}$ for all $1\le i<j\le m$). Moreover, we say that $\0$ \textit{is monotone on maximal chains in $X$} if, for any maximal chain $C$ in $X$, $\0$ is increasing or decreasing on $C$. 
	
	\subsection{Centralizers of elements of $I(X,K)$}
	
	Given a subset $S\sst I(X,K)$, we denote by $\gen {S}$ the $K$-subspace of $I(X,K)$ generated by $S$.\footnote{Observe that in~\cite{FKS} the notation $\gen {S}$ was used for the \emph{ideal} generated by $S\sst I(X,K)$.} 
	
	\begin{lemma}\cite[Lemma 5.9]{GK}\label{inter-C_I(e_A)=D+<e_xy>}
		For all $x<y$ from $X$ 
		\begin{align*}
			\bigcap_{Y\supseteq\{x,y\}}C_{I(X,K)}(e_Y)=D(X,K)\oplus\gen{e_{xy}}.
		\end{align*}
	\end{lemma}
	
	
	
	
	
	\begin{lemma}\label{basis-of-C_I(d)}
		For any $d\in D(X,K)$
		\begin{align*}
			C_{I(X,K)}(d)=D(X,K)\oplus\gen{e_{xy}\in\B:d(x,x)=d(y,y)}.
		\end{align*}
	\end{lemma}
	\begin{proof}
		Let $d\in D(X,K)$ and $f\in I(X,K)$. Then
		\begin{align*}
			[f,d] & = [f_D,d]+[f_J,d]=[f_J,d] = \left[\sum_{x<y}f(x,y)e_{xy}, \sum_{u\in X}d(u,u)e_u\right]\\
			& = \sum_{x<y}f(x,y)\left[e_{xy}, \sum_{u\in X}d(u,u)e_u\right] = \sum_{x<y}f(x,y)(d(y,y)-d(x,x))e_{xy}.
		\end{align*}
		Thus, $f\in C_{I(X,K)}(d)$ if and only if $d(y,y)=d(x,x)$ for all $x<y$ such that $f(x,y)\neq 0$.
	\end{proof}
	
	
	\section{Commutativity preservers of $I(X,K)$}\label{sec-comm-preservers}
	
	\subsection{The values of a commutativity preserver on the standard basis}
	
	\begin{proposition}\label{vf-on-the-natural-basis}
		Let $\vf:I(X,K)\to I(X,K)$ be a linear map. Then $\vf$ is a commutativity preserver if and only if
		\begin{align}
			[\vf(e_{xy}),\vf(e_{uv})]&=0\text{ for all }x\le y\text{ and }u\le v\text{ with }[e_{xy},e_{uv}]=0,\label{[vf-comm-pres-on-the-basis]}\\
			[\vf(e_x),\vf(e_{xy})]&=[\vf(e_{xz}),\vf(e_{zy})]=[\vf(e_{xy}),\vf(e_y)]\text{ for all }x<z<y.\label{[vf(e_xz)_vf(e_zy)]=const}
		\end{align}
	\end{proposition}
	\begin{proof}
		Let $\vf$ be a commutativity preserver. Then \cref{[vf-comm-pres-on-the-basis]} is immediate. Since $[e_x+e_y,e_{xy}]=0$, then $[\vf(e_x),\vf(e_{xy})]=[\vf(e_{xy}),\vf(e_y)]$, so it only remains to prove the first equality of \cref{[vf(e_xz)_vf(e_zy)]=const}. The latter follows from $[e_x+e_{xz},e_{zy}-e_{xy}]=0$. Indeed, using \cref{[vf-comm-pres-on-the-basis]} we have
		\begin{align*}
			0&=[\vf(e_x),\vf(e_{zy})]-[\vf(e_x),\vf(e_{xy})]+[\vf(e_{xz}),\vf(e_{zy})]-[\vf(e_{xz}),\vf(e_{xy})]\\
			&=[\vf(e_{xz}),\vf(e_{zy})]-[\vf(e_x),\vf(e_{xy})].
		\end{align*}
		
		Conversely, assume \cref{[vf-comm-pres-on-the-basis],[vf(e_xz)_vf(e_zy)]=const}. Take arbitrary $f,g\in I(X,K)$ and calculate
		\begin{align*}
			[\vf(f),\vf(g)]&=\sum_{x\le y}\sum_{u\le v}f(x,y)g(u,v)[\vf(e_{xy}),\vf(e_{uv})]\\
			&=\sum_{x\le z\le y}(f(x,z)g(z,y)[\vf(e_{xz}),\vf(e_{zy})]+f(z,y)g(x,z)[\vf(e_{zy}),\vf(e_{xz})])\\
			&=\sum_{x\le z\le y}(f(x,z)g(z,y)-g(x,z)f(z,y))[\vf(e_{xz}),\vf(e_{zy})]\\
			&=\sum_{x<y}\left(\sum_{x\le z\le y}(f(x,z)g(z,y)-g(x,z)f(z,y)\right)[\vf(e_x),\vf(e_{xy})]\\
			&=\sum_{x<y}[f,g](x,y)[\vf(e_x),\vf(e_{xy})].
		\end{align*}
		So, if $[f,g]=0$, then $[\vf(f),\vf(g)]=0$.
	\end{proof}
	
	\begin{remark}\label{[vf(e_x)_vf(e_xy)]-LI}
		Let $\vf:I(X,K)\to I(X,K)$ be a commutativity preserver. If $\{[\vf(e_x),\vf(e_{xy})]: x<y\}$ is linearly independent, then $\vf$ is strong.
	\end{remark}
	
	\subsection{Commutativity preservers of shift type}
	
	\begin{definition}
		Given a linear map $\af:I(X,K)\to D(X,K)$, define $S_\af:I(X,K)\to I(X,K)$ by means of $S_\af(f)=f+\af(f)$ for all $f\in I(X,K)$.
	\end{definition}
	
	\begin{proposition}\label{S_af-strong-comm-pres}
		The map $S_\af$ is a commutativity preserver of $I(X,K)$ if and only if $\af:I(X,K)\to D(X,K)$ satisfies
		\begin{align}
			[\af(e_{xy}),e_{uv}]&=0\text{ for all }e_{xy}\ne e_{uv}\text{ in }\B,\label{[af(e_xy)_e_uv]=0}\\
			[\af(e_z),e_{xy}]&=0\text{ for all }e_{xy}\in\B\text{ such that }z\not\in\{x,y\}\text{ or }l(x,y)>1,\label{[af(e_xy)_e_z]=0}\\
			[\af(e_x),e_{xy}]&=[e_{xy},\af(e_y)]\text{ for all }e_{xy}\in\B\text{ such that }l(x,y)=1.\label{[af(e_xy)_e_y]=[e_x_af(e_y)]}
		\end{align}
		Moreover, $S_\af$ is strong if and only if 
		\begin{align}\label{[af(e_x)_e_xy]ne-e_xy}
			[\af(e_x),e_{xy}]\ne -e_{xy}\text{ for all }e_{xy}\in\B\text{ such that }l(x,y)=1,
		\end{align}
		in which case $S_\af$ is bijective if and only if 
		\begin{align}\label{af(dl)-ne-dl}
			\af(\dl)\ne-\dl.
		\end{align}
		
	\end{proposition}
	\begin{proof}
		We first observe that, for arbitrary $f,g\in I(X,K)$,
		\begin{align}\label{[S_af(f)_S_af(g)]=[f_g]+[af(f)_g]+[f_af(g)]}
			[S_\af(f),S_\af(g)]=[f,g]+[\af(f),g]+[f,\af(g)],
		\end{align}
		because $[\af(f),\af(g)]=0$. 
		
		Assume \cref{[af(e_xy)_e_uv]=0,[af(e_xy)_e_z]=0,[af(e_xy)_e_y]=[e_x_af(e_y)]}. Then
		\begin{align*}
			[\af(f),g]&=\sum_{x\le y}\sum_{u\le v}f(x,y)g(u,v)[\af(e_{xy}),e_{uv}]\\
			&=\sum_{x<y}\sum_{u<v}f(x,y)g(u,v)[\af(e_{xy}),e_{uv}]+\sum_{z\in X}\sum_{x<y}f(z,z)g(x,y)[\af(e_z),e_{xy}]\\
			&=\sum_{x<y}f(x,y)g(x,y)[\af(e_{xy}),e_{xy}]+\sum_{l(x,y)=1}g(x,y)(f(x,x)-f(y,y))[\af(e_x),e_{xy}].
		\end{align*}
		By symmetry
		\begin{align*}
			[f,\af(g)]&=-[\af(g),f]\\
			&=-\sum_{x\le y}g(x,y)f(x,y)[\af(e_{xy}),e_{xy}]-\sum_{l(x,y)=1}f(x,y)(g(x,x)-g(y,y))[\af(e_x),e_{xy}].
		\end{align*}
		Notice that for all $x<y$ with $l(x,y)=1$ one has
		\begin{align*}
			[f,g](x,y)&=f(x,x)g(x,y)+f(x,y)g(y,y)-g(x,x)f(x,y)-g(x,y)f(y,y)\\
			&=f(x,y)(g(y,y)-g(x,x))+g(x,y)(f(x,x)-f(y,y)).
		\end{align*}
		Thus, 
		\begin{align}\label{[S_af(f)_S_af(g)]=[f_g]+[f_g](x_y)[af(e_x)_e_xy]}
			[S_\af(f),S_\af(g)]=[f,g]+\sum_{l(x,y)=1}[f,g](x,y)[\af(e_x),e_{xy}],  
		\end{align}
		and obviously $S_\af$ is a commutativity preserver. 
		
		Conversely, assume that $S_\af$ is a commutativity preserver. Take arbitrary $e_{xy}\ne e_{uv}$ in $\B$.
		
		\textit{Case 1.} $[e_{xy},e_{uv}]=0$. Then $[S_\af(e_{xy}),S_\af(e_{uv})]=0$, so
		\begin{align*}
			[\af(e_{xy}),e_{uv}]+[e_{xy},\af(e_{uv})]=0
		\end{align*}
		by \cref{[S_af(f)_S_af(g)]=[f_g]+[af(f)_g]+[f_af(g)]}. However, $[\af(e_{xy}),e_{uv}]\in\gen{e_{uv}}$ and $[e_{xy},\af(e_{uv})]\in\gen{e_{xy}}$. Since $e_{xy}\ne e_{uv}$, we conclude that $[\af(e_{xy}),e_{uv}]=[e_{xy},\af(e_{uv})]=0$.
		
		\textit{Case 2.} $y=u$, so that $[e_{xy},e_{uv}]=e_{xv}$. Then $[e_{xy}-e_{xv},e_{yv}+e_v]=0$. Hence, by \cref{[S_af(f)_S_af(g)]=[f_g]+[af(f)_g]+[f_af(g)]} and the result of Case 1,
		\begin{align*}
			0&=[\af(e_{xy})-\af(e_{xv}),e_{yv}+e_v]+[e_{xy}-e_{xv},\af(e_{yv})+\af(e_v)]\\
			&=[\af(e_{xy}),e_{yv}]+[e_{xy},\af(e_{yv})]+[e_{xy},\af(e_v)]-[e_{xv},\af(e_v)].
		\end{align*}
		Since $[\af(e_{xy}),e_{yv}]\in\gen{e_{yv}}$, $[e_{xy},\af(e_{yv})]+[e_{xy},\af(e_v)]\in\gen{e_{xy}}$ and $[e_{xv},\af(e_v)]\in\gen{e_{xv}}$, we conclude that 
		\begin{align}\label{[e_xv_af(e_v)]=0}
			[\af(e_{xy}),e_{yv}]=[e_{xy},\af(e_{yv})]+[e_{xy},\af(e_v)]=[e_{xv},\af(e_v)]=0.
		\end{align}
		
		\textit{Case 3.} $x=v$, so that $[e_{xy},e_{uv}]=-e_{uy}$. Then $[e_u+e_{ux},e_{xy}-e_{uy}]=0$. Using the results of Cases 1 and 2 we have
		\begin{align*}
			0&=[\af(e_u)+\af(e_{ux}),e_{xy}-e_{uy}]+[e_u+e_{ux},\af(e_{xy})-\af(e_{uy})]\\
			&=[\af(e_u),e_{xy}]-[\af(e_u),e_{uy}]+[e_{ux},\af(e_{xy})].
		\end{align*}
		Therefore, 
		\begin{align}\label{[af(e_u)_e_uy]=0}
			[\af(e_u),e_{xy}]=[\af(e_u),e_{uy}]=[e_{ux},\af(e_{xy})]=0.
		\end{align}
		We have thus proved \cref{[af(e_xy)_e_uv]=0}. 
		
		We proceed with the proof of \cref{[af(e_xy)_e_z]=0}. Let $e_{xy}\in\B$ and $z\not\in\{x,y\}$. Then $[e_z,e_{xy}]=0$, so \cref{[af(e_xy)_e_z]=0} follows from \cref{[S_af(f)_S_af(g)]=[f_g]+[af(f)_g]+[f_af(g)]} and the commutativity preserving property of $S_\af$. Now let $e_{xy}\in\B$ with $l(x,y)>1$. We need to prove that $[\af(e_x),e_{xy}]=[\af(e_y),e_{xy}]=0$. But we have already seen this in \cref{[e_xv_af(e_v)]=0,[af(e_u)_e_uy]=0}. Finally, \cref{[af(e_xy)_e_y]=[e_x_af(e_y)]} is explained by $[e_x+e_y,e_{xy}]=0$.
		
		Let us pass to the proof of the second statement of the proposition. Assume that $S_\af$ is a commutativity preserver satisfying \cref{[af(e_x)_e_xy]ne-e_xy}. If $[S_\af(f),S_\af(g)]=0$, then $[f,g](x,y)=0$ for all $x<y$ with $l(x,y)>1$ by \cref{[S_af(f)_S_af(g)]=[f_g]+[f_g](x_y)[af(e_x)_e_xy]}. Therefore, \cref{[S_af(f)_S_af(g)]=[f_g]+[f_g](x_y)[af(e_x)_e_xy]} implies $[f,g](x,y)(e_{xy}+[\af(e_x),e_{xy}])=0$ for all $x<y$ with $l(x,y)=1$, whence $[f,g](x,y)=0$ for such $x<y$ thanks to \cref{[af(e_x)_e_xy]ne-e_xy}. Conversely, if $S_\af$ is strong, then $[S_\af(e_x),S_\af(e_{xy})]\ne 0$ because $[e_x,e_{xy}]=e_{xy}\ne 0$. Hence, 
		\begin{align*}
			[e_x,e_{xy}]+[\af(e_x),e_{xy}]\ne 0
		\end{align*}
		by \cref{[S_af(f)_S_af(g)]=[f_g]+[af(f)_g]+[f_af(g)]}. This implies \cref{[af(e_x)_e_xy]ne-e_xy}.
		
		Finally, assume that $S_\af$ is a strong commutativity preserver. Since $S_\af(\dl)=\dl+\af(\dl)$, then $\af(\dl)\ne-\dl\iff S_\af(\dl)\ne 0$. The latter is equivalent to injectivity of $S_\af$ (and hence, bijectivity of $S_\af$, because $|X|<\infty$) by \cref{ker-vf-strong}.
	\end{proof}
	
	\begin{definition}
		A strong commutativity preserver of the form $S_\af$, where $\af$ satisfies \cref{[af(e_xy)_e_uv]=0,[af(e_xy)_e_z]=0,[af(e_xy)_e_y]=[e_x_af(e_y)],[af(e_x)_e_xy]ne-e_xy}, is said to be \textit{of shift type}. It is a generalization of a shift map defined in \cite[Definition 5.24]{GK}.
	\end{definition}

	\section{Linear maps preserving the commutativity and $D(X,K)$}\label{sec-diag-pres}
	
	\begin{definition}
		Let $\vf:I(X,K)\to I(X,K)$ be a linear map. We say that $\vf$ {\it preserves the diagonality} (or is a {\it diagonality preserver}), if $\vf(D(X,K))\sst D(X,K)$. A diagonality preserver is said to be {\it strong}, if $\vf(D(X,K))=D(X,K)$.
	\end{definition}
	
	Since $|X|<\infty$, then any bijective diagonality preserver is clearly strong. Observe that $S_\af$ is a diagonality preserver for any $\af$. 
	
	\subsection{From $\vf$ to $(\0,\sg,\nu,c)$}\label{sec-from-(0_sg_nu_c)-to-vf}
	
	\begin{lemma}\label{vf(<e_xy>+D)=<0(e_xy)>+D}
		Let $\vf:I(X,K)\to I(X,K)$ be a bijective strong commutativity preserver. If $\vf$ additionally preserves the diagonality, then for any $x<y$ we have $\vf(\gen{e_{xy}}\oplus D(X,K))=\gen{\0(e_{xy})}\oplus D(X,K)$ for some bijection $\0=\0_\vf:\B\to\B$. 
	\end{lemma}
	\begin{proof}
		By \cref{inter-C_I(e_A)=D+<e_xy>,basis-of-C_I(d)} and \cite[Lemmas 1.1--1.2]{GK}
		\begin{align*}
			\vf(\gen{e_{xy}}\oplus D(X,K))&=\bigcap_{Y\supseteq\{x,y\}}\vf(C_{I(X,K)}(e_Y))
			=\bigcap_{Y\supseteq\{x,y\}}C_{I(X,K)}(\vf(e_Y))\\
			&=\bigcap_{Y\supseteq\{x,y\}}\left(D(X,K)\oplus\gen{e_{uv}\in\B:\vf(e_Y)(u,u)=\vf(e_Y)(v,v)}\right)\\
			&=D(X,K)\oplus\bigcap_{Y\supseteq\{x,y\}}\gen{e_{uv}\in\B:\vf(e_Y)(u,u)=\vf(e_Y)(v,v)}\\
			&=D(X,K)\oplus\gen{e_{uv}\in\B:\ \vf(e_Y)(u,u)=\vf(e_Y)(v,v),\forall Y\supseteq\{x,y\}}.
		\end{align*}
		Since $\dim(\vf(\gen{e_{xy}}\oplus D(X,K)))=\dim(\gen{e_{xy}}\oplus D(X,K))$, there exists a unique pair $u<v$ such that $\vf(e_Y)(u,u)=\vf(e_Y)(v,v)$ for all $Y\supseteq\{x,y\}$. So, we define $\0(e_{xy})=e_{uv}$. Let us prove that $\0$ is injective. Assume that $\0(e_{xy})=\0(e_{x'y'})=e_{uv}$. Then $\vf(e_{xy})=ke_{uv}+d$ and $\vf(e_{x'y'})=k'e_{uv}+d'$, where $k,k'\in K$ and $d,d'\in D(X,K)$. Therefore, $\vf(k'e_{xy}-ke_{x'y'})=k'd-kd'\in D(X,K)$. But $\vf$ is strong, so $k'e_{xy}-ke_{x'y'}\in D(X,K)$, whence $e_{xy}=e_{x'y'}$ (and $k=k'$). Since $\B$ is finite, $\0$ is bijective.
	\end{proof}
	Throughout this subsection, $\vf$ always denotes a bijection which preserves diagonality and is a strong commutativity preserver.
	For any $x<y$ we write
	\begin{align}\label{vf(e_xy)=0(e_xy)+nu(e_xy)}
		\vf(e_{xy})=\sg(x,y)\0(e_{xy})+\nu(e_{xy}),
	\end{align}
	where $\sg:X^2_<\to K^*$ and $\nu:\B\to D(X,K)$.
	
	\begin{lemma}\label{0-strong-comm-pres}
		The bijection $\0:\B\to\B$ is a strong commutativity preserver.
	\end{lemma}
	\begin{proof}
		Let $e_{xy},e_{uv}\in\B$ such that $[e_{xy},e_{uv}]=0$ and $e_{xy}\ne e_{uv}$. Then by \cref{vf(e_xy)=0(e_xy)+nu(e_xy)}
		\begin{align*}
			0=[\vf(e_{xy}),\vf(e_{uv})]&=[\sg(x,y)\0(e_{xy})+\nu(e_{xy}),\sg(u,v)\0(e_{uv})+\nu(e_{uv})]\\
			&=\sg(x,y)\sg(u,v)[\0(e_{xy}),\0(e_{uv})]+\sg(x,y)[\0(e_{xy}),\nu(e_{uv})]\\
			&\quad+\sg(u,v)[\nu(e_{xy}),\0(e_{uv})].
		\end{align*}
		Clearly, $[\0(e_{xy}),\nu(e_{uv})]\in\gen{\0(e_{xy})}$, $[\nu(e_{xy}),\0(e_{uv})]\in\gen{\0(e_{uv})}$, while $[\0(e_{xy}),\0(e_{uv})]$ is either $0$ or $\pm e_{ab}$ for some $e_{ab}\not\in\{\0(e_{xy}),\0(e_{uv})\}$. By the linear independence of $\B$, we conclude that $[\0(e_{xy}),\0(e_{uv})]=0$ (and $[\0(e_{xy}),\nu(e_{uv})]=[\nu(e_{xy}),\0(e_{uv})]=0$).
		
		Denote by $\0^2$ the map $\B^2\to\B^2$ sending $(e_{xy},e_{uv})$ to $(\0(e_{xy}),\0(e_{uv}))$ and by  $\mathcal{S}$ the set of pairs $(e_{xy},e_{uv})\in\B^2$ such that $[e_{xy},e_{uv}]=0$. Then $\0^2$ is a bijection and $\0^2(\mathcal{S})\sst\mathcal{S}$. It follows that $\0^2(\mathcal{S})=\mathcal{S}$, because $\mathcal{S}$ is finite. Thus, $\0$ is strong.
	\end{proof}
	
	\begin{corollary}\label{[e_xy_e_uv]=0=>[0(e_xy)_nu(e_uv)]=0}
		If $[e_{xy},e_{uv}]=0$ and $e_{xy}\ne e_{uv}$, then $[\0(e_{xy}),\nu(e_{uv})]=0$.
	\end{corollary}
	\begin{proof}
		This is a consequence of the proof of \cref{0-strong-comm-pres}.
	\end{proof}
	
	\begin{corollary}\label{theta_e_xy_theta_e_yz}
		Let $x<y<z$ in $X$. Then there are $u<v<w$ in $X$ such that $\{\0(e_{xy}),\0(e_{yz})\}=\{e_{uv},e_{vw}\}$.
	\end{corollary}
	\begin{proof}
		Indeed, since $[e_{xy},e_{yz}]\ne 0$, then $[\0(e_{xy}),\0(e_{yz})]\ne 0$ by \cref{0-strong-comm-pres}, whence there are $u<v<w$ such that either $\0(e_{xy})=e_{uv}$ and $\0(e_{yz})=e_{vw}$, or $\0(e_{xy})=e_{vw}$ and $\0(e_{yz})=e_{uv}$.
	\end{proof}
	
	\begin{remark}\label{center-of-B}
		Observe that $Z(\B)=Z(J(I(X,K)))\cap\B$, so $Z(\B)=\{e_{xy}: x\in\Min(X)\text{ and }y\in\Max(X)\}$ by~\cite[Proposition 2.5]{FKS}.
	\end{remark}
	
	\begin{corollary}\label{theta_fix_center}
		We have $\0(Z(\B))=Z(\B)$.
	\end{corollary} 
	
	\begin{lemma}\label{theta_e_xz}
		Let $x<y<z$ in $X$. If $\0(e_{xy})=e_{uv}$ (resp.~$e_{vw}$) and $\0(e_{yz})=e_{vw}$ (resp.~$e_{uv}$), then $\0(e_{xz})=e_{uw}$.
	\end{lemma}
	\begin{proof}
		Suppose first that $\0(e_{xy})=e_{uv}$ and $\0(e_{yz})=e_{vw}$. Since $[e_x,e_{xy}]\ne 0$, then in view of \cref{vf(e_xy)=0(e_xy)+nu(e_xy)} and $\vf(e_x)\in D(X,K)$ we have $[\vf(e_x),\0(e_{xy})]\ne 0$, i.e. $\vf(e_x)(u,u)\ne \vf(e_x)(v,v)$. Similarly, from $[e_x,e_{yz}]=0$ we deduce $\vf(e_x)(v,v)=\vf(e_x)(w,w)$. It follows that 
		\begin{align}\label{vf(e_x)(uu)-ne-vf(e_x)(ww)}
			\vf(e_x)(u,u)\ne \vf(e_x)(w,w).
		\end{align}
		Now, considering $[e_z,e_{xy}]=0$ and $[e_z,e_{yz}]\ne 0$ we obtain
		\begin{align}\label{vf(e_z)(uu)-ne-vf(e_z)(ww)}
			\vf(e_z)(u,u)\ne \vf(e_z)(w,w).
		\end{align}
		
		Take $a<b$ such that $\0(e_{ab})=e_{uw}$. Assume that $x\not\in\{a,b\}$. Then $[e_x,e_{ab}]=0$ implies $\vf(e_x)(u,u)=\vf(e_x)(w,w)$ contradicting \cref{vf(e_x)(uu)-ne-vf(e_x)(ww)}. Hence, $x\in\{a,b\}$. If $b=x$, then $[e_{ab},e_{xy}]=e_{ay}\ne 0$, so $[e_{uw},e_{uv}]\ne 0$ by \cref{0-strong-comm-pres}, a contradiction. Therefore, $a=x$. Finally, assuming $b\ne z$ we get $[e_z,e_{ab}]=[e_z,e_{xb}]=0$, which yields $\vf(e_z)(u,u)=\vf(e_z)(w,w)$, a contradiction with \cref{vf(e_z)(uu)-ne-vf(e_z)(ww)}. Thus, $(a,b)=(x,z)$.
		
		In the case where $\0(e_{xy})=e_{vw}$ and $\0(e_{yz})=e_{uv}$ we also have the inequalities \cref{vf(e_x)(uu)-ne-vf(e_x)(ww),vf(e_z)(uu)-ne-vf(e_z)(ww)}, and the proof is analogous.
	\end{proof}
	
	\begin{corollary}\label{[0(e_xy)_0(e_yz)]=+-0(e_xz)}
		Let $x<y<z$ in $X$. Then either $[\0(e_{xy}),\0(e_{yz})]=\0(e_{xy})\0(e_{yz})=\0(e_{xz})$ or $[\0(e_{xy}),\0(e_{yz})]=-\0(e_{yz})\0(e_{xy})=-\0(e_{xz})$.
	\end{corollary}
	
	\begin{remark}\label{0_vf-inv}
		$\0_{\vf\m}=(\0_\vf)\m$ and $\sg_{\vf\m}(u,v)=\sg_\vf(x,y)\m$, where $\0_\vf(e_{xy}) = e_{uv}$. Indeed, if $\0_\vf(e_{xy}) = e_{uv}$, then 
		\begin{align*}
			e_{xy} & =\vf^{-1}(\vf(e_{xy}))=\vf^{-1}(\sg_{\vf}(x,y)\0_{\vf}(e_{xy})+\nu_{\vf}(e_{xy})) \\
			& = \sg_{\vf}(x,y)\vf^{-1}(\0_{\vf}(e_{xy}))+\vf^{-1}(\nu_{\vf}(e_{xy})) \\
			& = \sg_{\vf}(x,y)\sg_{\vf^{-1}}(u,v)\0_{\vf^{-1}}(\0_{\vf}(e_{xy}))+\sg_{\vf}(x,y)\nu_{\vf^{-1}}(e_{uv})+\vf^{-1}(\nu_{\vf}(e_{xy})).
		\end{align*}
		Since $\0_{\vf^{-1}}(\0_{\vf}(e_{xy}))\in\B$ and $\sg_{\vf}(x,y)\nu_{\vf^{-1}}(e_{uv})+\vf^{-1}(\nu_{\vf}(e_{xy}))\in D(X,K)$ we have $e_{xy}=\0_{\vf^{-1}}(\0_{\vf}(e_{xy}))$ and $\sg_{\vf}(x,y)\sg_{\vf^{-1}}(u,v)=1$.
	\end{remark}
	
	\begin{lemma}\label{0_monotone}
		The bijection $\0$ is monotone on maximal chains in $X$.
	\end{lemma}
	\begin{proof}
		Let $u_1<u_2<\dots<u_m$ be a maximal chain in $X$. The proof of the existence of a chain $v_1<v_2<\dots<v_m$ satisfying $\0(e_{u_iu_j})=e_{v_iv_j}$ for all $1\le i<j\le m$, or $\0(e_{u_iu_j})=e_{v_{m-j+1}v_{m-i+1}}$ for all $1\le i<j\le m$, is exactly the same as the proof of \cite[Lemma 5.4]{FKS}, just replacing \cite[Lemma 5.2]{FKS} by \cref{theta_e_xz}. It remains to prove that the chain $v_1<v_2<\dots<v_m$ is maximal. 
		
		Since $u_1\in \Min(X)$ and $u_m\in \Max(X)$, then $e_{u_1u_m}\in Z(\B)$. Hence $e_{v_1v_m}\in Z(\B)$, by \cref{theta_fix_center}. Thus, $v_1\in \Min(X)$ and $v_m\in \Max(X)$, by \cref{center-of-B}. Now, let $i\in \{1,\ldots,m-1\}$ and consider $u_i<u_{i+1}$. Suppose $\0(e_{u_iu_{i+1}})=e_{v_iv_{i+1}}$. If there were $v\in X$ such that $v_i<v<v_{i+1}$, then there would be $w_i<w<w_{i+1}$ such that $\0\m(e_{v_iv})=e_{w_iw}$ and $\0\m(e_{vv_{i+1}})=e_{ww_{i+1}}$, or $\0\m(e_{v_iv})=e_{ww_{i+1}}$ and $\0\m(e_{vv_{i+1}})=e_{w_iw}$, by \cref{theta_e_xy_theta_e_yz,0_vf-inv}. Consequently, $\0\m(e_{v_iv_{i+1}})=e_{w_iw_{i+1}}$, by \cref{theta_e_xz}, and so $\0(e_{w_iw_{i+1}})=\0(e_{u_iu_{i+1}})$. Thus $u_i=w_i$ and $u_{i+1}=w_{i+1}$, and so $u_i<w<u_{i+1}$, which would contradict the maximality of $u_1<u_2<\dots<u_m$. The case $\0(e_{u_iu_{i+1}})=e_{v_{m-i}v_{m-i+1}}$ is analogous. 
	\end{proof}

	\begin{lemma}\label{[0(e_xy)_nu(e_yz)]=[nu(e_xy)_0(e_yz)]=0}
		Let $x<y<z$. Then $[\0(e_{xy}),\nu(e_{yz})]=0$ and $[\nu(e_{xy}),\0(e_{yz})]=0$.
	\end{lemma}
	\begin{proof}
		Applying $\vf$ to the product $[e_x+e_{xy},e_{yz}-e_{xz}]=0$ and using \cref{vf(e_xy)=0(e_xy)+nu(e_xy)} we get
		\begin{align*}
			0&=[\vf(e_x)+\sg(x,y)\0(e_{xy})+\nu(e_{xy}),\sg(y,z)\0(e_{yz})+\nu(e_{yz})-\sg(x,z)\0(e_{xz})-\nu(e_{xz})]\\
			&=\sg(y,z)[\vf(e_x),\0(e_{yz})]+\sg(y,z)[\nu(e_{xy}),\0(e_{yz})]-\sg(x,z)[\vf(e_x)+\nu(e_{xy}),\0(e_{xz})]\\
			&\quad+\sg(x,y)[\0(e_{xy}),\nu(e_{yz})]-\sg(x,y)[\0(e_{xy}),\nu(e_{xz})]\\
			&\quad+\sg(x,y)\sg(y,z)[\0(e_{xy}),\0(e_{yz})]-\sg(x,y)\sg(x,z)[\0(e_{xy}),\0(e_{xz})].
		\end{align*}
		Observe that $[\0(e_{xy}),\0(e_{xz})]=0$ by \cref{0-strong-comm-pres}, $[\0(e_{xy}),\nu(e_{xz})]=0$ by \cref{[e_xy_e_uv]=0=>[0(e_xy)_nu(e_uv)]=0} and $[\0(e_{xy}),\0(e_{yz})]=\pm \0(e_{xz})$ by \cref{[0(e_xy)_0(e_yz)]=+-0(e_xz)}. 
		Moreover, it follows from $[e_x,e_{yz}]=0$ that $[\vf(e_x),\0(e_{yz})]=0$. Thus, 
		\begin{align}
			0 &=\sg(x,y)[\0(e_{xy}),\nu(e_{yz})]+\sg(y,z)[\nu(e_{xy}),\0(e_{yz})]\label{elem-of-<0(e_xy)>+elem-of-<0(e_yz)>}\\
			& \quad -\sg(x,z)[\vf(e_x)+\nu(e_{xy}),\0(e_{xz})]\pm\sg(x,y)\sg(y,z)\0(e_{xz}),\label{elem-of-<0(e_xz)>}
		\end{align}
		where $[\0(e_{xy}),\nu(e_{yz})] \in \gen{\0(e_{xy})}$, $[\nu(e_{xy}),\0(e_{yz})]\in \gen{\0(e_{yz})}$ and both appear in \cref{elem-of-<0(e_xy)>+elem-of-<0(e_yz)>} with non-zero coefficients, while \cref{elem-of-<0(e_xz)>} belongs to $\gen{\0(e_{xz})}$. It remains to use the linear independence of $\{\0(e_{xy}),\0(e_{yz}),\0(e_{xz})\}$.
	\end{proof}
	
	\begin{lemma}\label{[0(e_xy)_nu(e_xy)]=0}
		Let $x<y$ such that $\{x,y\}$ is not a maximal chain in $X$. Then $[\0(e_{xy}),\nu(e_{xy})]=0$.
	\end{lemma}
	\begin{proof}
		There exists $z\in X$ such that either $z<x$ or $x<z<y$ or $y<z$. 
		
		{\it Case 1.} $z<x$. By \cref{theta_e_xy_theta_e_yz} there are $u<v<w$ such that  $\{\0(e_{zx}),\0(e_{xy})\}=\{e_{uv},e_{vw}\}$. Consider the case $\0(e_{zx})=e_{uv}$ and $\0(e_{xy})=e_{vw}$. Then $\0(e_{zy})=e_{uw}$ by \cref{theta_e_xz}. Now, $[\0(e_{zx}),\nu(e_{xy})]=0$ by \cref{[0(e_xy)_nu(e_yz)]=[nu(e_xy)_0(e_yz)]=0} and $[\0(e_{zy}),\nu(e_{xy})]=0$ by \cref{[e_xy_e_uv]=0=>[0(e_xy)_nu(e_uv)]=0}. It follows from \cref{basis-of-C_I(d)} that $\nu(e_{xy})(u,u)=\nu(e_{xy})(v,v)$ and $\nu(e_{xy})(u,u)=\nu(e_{xy})(w,w)$, whence $\nu(e_{xy})(v,v)=\nu(e_{xy})(w,w)$. The latter is equivalent to $[\0(e_{xy}),\nu(e_{xy})]=0$, by \cref{basis-of-C_I(d)}. The case $\0(e_{zx})=e_{vw}$ and $\0(e_{xy})=e_{uv}$ is analogous.
		
		{\it Case 2.} $x<z<y$. As above, $\{\0(e_{xz}),\0(e_{zy})\}=\{e_{uv},e_{vw}\}$ for some $u<v<w$. Assume $\0(e_{xz})=e_{uv}$ and $\0(e_{zy})=e_{vw}$, so that $\0(e_{xy})=e_{uw}$. Then $[\0(e_{xz}),\nu(e_{xy})]=0$ implies $\nu(e_{xy})(u,u)=\nu(e_{xy})(v,v)$ and $[\0(e_{zy}),\nu(e_{xy})]=0$ implies $\nu(e_{xy})(v,v)=\nu(e_{xy})(w,w)$. Hence, $\nu(e_{xy})(u,u)=\nu(e_{xy})(w,w)$, that is $[\0(e_{xy}),\nu(e_{xy})]=0$. If $\0(e_{xz})=e_{vw}$ and $\0(e_{zy})=e_{uv}$, the proof is analogous.
		
		{\it Case 3.} $z>y$. This case is similar to Case 1.
	\end{proof}
	
	\begin{proposition}\label{nu-almost-central-valued}
		Let $x<y$ such that $\{x,y\}$ is not a maximal chain in $X$. Then $\nu(e_{xy})\in Z(I(X,K))$.
	\end{proposition}
	\begin{proof}
		Since $\nu(e_{xy})\in D(X,K)$ and $\0$ is a bijection, the result follows from \cref{[e_xy_e_uv]=0=>[0(e_xy)_nu(e_uv)]=0,[0(e_xy)_nu(e_yz)]=[nu(e_xy)_0(e_yz)]=0,[0(e_xy)_nu(e_xy)]=0}.
	\end{proof}
	
	The following example shows, in particular, that the condition that $\{x,y\}$ is not a maximal chain in $X$ cannot be dropped from \cref{nu-almost-central-valued}.
	\begin{example}
		Let $X=\{1,2,3\}$ with the following Hasse diagram.
		
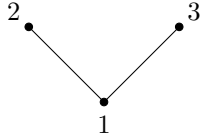
\begin{center}
			\begin{tikzpicture}
			\draw  (0,0)-- (-1,1);
			\draw  (0,0)-- (1,1);
			\draw [fill=black] (-1,1) circle (0.05);
			\draw  (-1.2,1.2) node {$2$};
			\draw [fill=black] (1,1) circle (0.05);
			\draw  (1.2,1.2) node {$3$};
			\draw [fill=black] (0,0) circle (0.05);
			\draw  (0,-0.3) node {$1$};
			\end{tikzpicture}
		\end{center}
		Consider the $K$-linear map $I(X,K)\to I(X,K)$ defined as follows: $\vf(e_1)=e_1$, $\vf(e_{12})=e_{12}+e_1+e_3$, $\vf(e_{13})=e_{13}+e_1+e_2$, $\vf(e_2)=e_2$, $\vf(e_3)=e_3$. Then $\vf$ is a bijective strong commutativity preserver of $I(X,K)$.
		
		Indeed, 
		\begin{align*}
			0&=[\vf(e_1),\vf(e_2)]=[\vf(e_1),\vf(e_3)]=[\vf(e_{12}),\vf(e_{13})]\\
			&=[\vf(e_{12}),\vf(e_3)]=[\vf(e_{13}),\vf(e_2)]=[\vf(e_2),\vf(e_3)]
		\end{align*}
		and
		\begin{align*}
			[\vf(e_1),\vf(e_{12})]&=e_{12}=[\vf(e_{12}),\vf(e_2)],\\
			[\vf(e_1),\vf(e_{13})]&=e_{13}=[\vf(e_{13}),\vf(e_3)],
		\end{align*}
		so by \cref{vf-on-the-natural-basis,[vf(e_x)_vf(e_xy)]-LI} the map $\vf$ is a strong commutativity preserver. It is clearly injective by \cref{ker-vf-strong}, and thus bijective. 
		
		Observe that $\vf(D(X,K))=D(X,K)$ and the corresponding $\0$ and $\nu$ are as follows: $\0=\id_{\B}$, $\nu(e_{12})=e_1+e_3$ and $\nu(e_{13})=e_1+e_2$. In particular, we see that $\nu(e_{12}),\nu(e_{13})\not\in Z(I(X,K))$. Moreover, if we compose $\vf$ with the conjugation by $f=\dl+e_{12}$, then the resulting bijective strong commutativity preservers $e_{12}$ to $e_1+e_3$, because
		\begin{align*}
			(\dl+e_{12})\vf(e_{12})(\dl+e_{12})\m&=(\dl+e_{12})(e_{12}+e_1+e_3)(\dl-e_{12})\\
			&=(e_{12}+e_1+e_3)(\dl-e_{12})\\
			&=e_{12}+e_1+e_3-e_{12}=e_1+e_3.
		\end{align*}
		This shows that \cref{vf(e_xy)=0(e_xy)+nu(e_xy)} does not hold without the diagonality preserving condition.
	\end{example}
	
	
	\begin{definition}
		Let $\vf:I(X,K)\to I(X,K)$ be a bijective strong commutativity preserver that also preserves the diagonality. We say that $\vf$ is \textit{pure} if $\nu(e_{xy})=0$ in \cref{vf(e_xy)=0(e_xy)+nu(e_xy)} for all $x<y$.
	\end{definition} 
	
	\begin{lemma}\label{S_af-circ-vf-pure}
		Let $\vf:I(X,K)\to I(X,K)$ be a bijective strong commutativity preserver that also preserves the diagonality. If we define $\af:I(X,K)\to D(X,K)$ by means of
		\begin{align}
			\af|_{D(X,K)}&=0,\label{af_D(X_K)=0}\\
			\af(\0(e_{xy}))&=\sg(x,y)\m\nu(e_{xy}),\ x<y,\label{af(0(e_xy))=nu(e_xy)}
		\end{align}
		then $S_\af$ is a bijective strong commutativity preserver. Moreover, $S\m_\af\circ\vf$ is pure.
	\end{lemma}
	\begin{proof}
		Due to \cref{S_af-strong-comm-pres} we only need to show that $\af$ satisfies \cref{[af(e_xy)_e_uv]=0}. For all $e_{xy},e_{uv}\in\B$, by \cref{af(0(e_xy))=nu(e_xy)} we have $[\af(\0(e_{xy})),\0(e_{uv})]=\sg(x,y)\m[\nu(e_{xy}),\0(e_{uv})]$ which is zero whenever $e_{xy}\ne e_{uv}$ in view of \cref{[e_xy_e_uv]=0=>[0(e_xy)_nu(e_uv)]=0,[0(e_xy)_nu(e_yz)]=[nu(e_xy)_0(e_yz)]=0}, whence \cref{[af(e_xy)_e_uv]=0}. Thanks to \cref{af_D(X_K)=0,af(dl)-ne-dl}, $S_\af$ is bijective. In fact, $S\m_\af=S_{-\af}$, so by \cref{vf(e_xy)=0(e_xy)+nu(e_xy),af_D(X_K)=0,af(0(e_xy))=nu(e_xy)}
		\begin{align*}
			(S\m_\af\circ\vf)(e_{xy})&=S_{-\af}(\sg(x,y)\0(e_{xy})+\nu(e_{xy}))=\sg(x,y)S_{-\af}(\0(e_{xy}))+\nu(e_{xy})\\
			&=\sg(x,y)(\0(e_{xy})-\sg(x,y)\m\nu(e_{xy}))+\nu(e_{xy})=\sg(x,y)\0(e_{xy}).
		\end{align*}
		Thus, $S\m_\af\circ\vf$ is pure.
	\end{proof}
	
	So, from now on, composing, if necessary, $\vf$ with a bijective strong commutativity preserver of shift type, we may assume for simplicity that $\vf$ is pure, although some results below hold for non-pure bijective commutativity preservers of $I(X,K)$.
	
	\begin{definition}\label{cxy}
		Fixed $\varphi$, we attach to it a map $c:X^2_{<}\to K$ defined by 
		\begin{align}\label{c(xy)=vf(e_x)(uu)-vf(e_x)(vv)}
			c(x,y)=\vf(e_x)(u,u)-\vf(e_x)(v,v),
		\end{align}
		where $\0(e_{xy})=e_{uv}$.
	\end{definition}
	
	\begin{lemma}\label{vf(e_x)(u_u)-vf(e_x)(v_v)-and-vf(e_y)(u_u)-vf(e_y)(v_v)}
		For all $x<y$ we have $c(x,y)\in K^*$. Moreover, if $\0(e_{xy})=e_{uv}$, then for all $z\in X$
		\begin{align}\label{vf(e_z)(u_u)-vf(e_z)(v_v)}
			[\vf(e_z),\0(e_{xy})](u,v)=\vf(e_z)(u,u)-\vf(e_z)(v,v)=
			\begin{cases}
				c(x,y), & z=x,\\
				-c(x,y), & z=y,\\
				0, & z\not\in\{x,y\}.
			\end{cases}
		\end{align}
	\end{lemma}
	\begin{proof}
		We are going to prove \cref{vf(e_z)(u_u)-vf(e_z)(v_v)}, and it will follow from the proof that $c(x,y)\in K^*$. Observe that 
		\begin{align}\label{[vf(e_z)_0(e_xy)]=(vf(e_z)(u_u)-vf(e_z)(v_v))e_uv}
			[\vf(e_z),\0(e_{xy})]=[\vf(e_z),e_{uv}]=(\vf(e_z)(u,u)-\vf(e_z)(v,v))e_{uv}
		\end{align}
		because $\vf(e_z)\in D(X,K)$. So the first equality of \cref{vf(e_z)(u_u)-vf(e_z)(v_v)} holds.
		
		{\it Case 1.} $z=x$. Since $[e_z,e_{xy}]=[e_x,e_{xy}]\neq 0$, we have $[\vf(e_z),\vf(e_{xy})]\neq 0$. Thus, by \cref{vf(e_xy)=0(e_xy)+nu(e_xy),[vf(e_z)_0(e_xy)]=(vf(e_z)(u_u)-vf(e_z)(v_v))e_uv,c(xy)=vf(e_x)(uu)-vf(e_x)(vv)}, 
		\begin{align*}
			0\ne [\vf(e_z),\0(e_{xy})](u,v)=\vf(e_x)(u,u)-\vf(e_x)(v,v)=c(x,y). 
		\end{align*}
		
		{\it Case 2.} $z=y$. Since $[e_x+e_y,e_{xy}]=0$, then $[\vf(e_x)+\vf(e_y),\vf(e_{xy})]=0$, so
		\begin{align*}
			[\vf(e_z),\vf(e_{xy})]=[\vf(e_y),\vf(e_{xy})]=-[\vf(e_x),\vf(e_{xy})].
		\end{align*}
		It follows from \cref{vf(e_xy)=0(e_xy)+nu(e_xy)} and the result of Case 1 that
		\begin{align*}
			[\vf(e_z),\0(e_{xy})](u,v)=-[\vf(e_x),\0(e_{xy})](u,v)=-c(x,y).
		\end{align*}
		
		{\it Case 3.} $z\not\in\{x,y\}$. Then $[e_z,e_{xy}]=0$, so $[\vf(e_z),\vf(e_{xy})]=0$. Therefore, $[\vf(e_z),\0(e_{xy})]=0$.
	\end{proof}
	
	\begin{lemma}\label{defn-of-c(x_z)}
		If $x<y<z$, then
		\begin{align}
			\0(e_{xz})=\0(e_{xy})\0(e_{yz})&\impl c(x,z)\sg(x,z)= \sg(x,y)\sg(y,z),\label{c(x_z)sg(x_z)=sg(x_y)sg(y_z)}\\
			\0(e_{xz})=\0(e_{yz})\0(e_{xy})&\impl c(x,z)\sg(x,z)=-\sg(x,y)\sg(y,z).\label{c(x_z)sg(x_z)=-sg(x_y)sg(y_z)}
		\end{align}
	\end{lemma}
	\begin{proof}
		By \cref{elem-of-<0(e_xy)>+elem-of-<0(e_yz)>,elem-of-<0(e_xz)>} and \cref{[e_xy_e_uv]=0=>[0(e_xy)_nu(e_uv)]=0} we have
		\begin{align*}
			-\sg(x,z)[\vf(e_x),\0(e_{xz})]\pm\sg(x,y)\sg(y,z)\0(e_{xz})=0
		\end{align*}
		depending on whether $\0(e_{xz})=\0(e_{xy})\0(e_{yz})$ or $\0(e_{xz})=\0(e_{yz})\0(e_{xy})$. So, thanks to \cref{vf(e_z)(u_u)-vf(e_z)(v_v)}, $c(x,z)\sg(x,z)=\pm\sg(x,y)\sg(y,z)$.
	\end{proof}
	
	\begin{lemma}\label{c(x_y)=c(u_y)=c(x_v)}
		If $x<y$, then
		\begin{enumerate}
			\item for all $u<x$ we have $c(u,y)=c(x,y)$;\label{c(u_y)=c(x_y)}
			\item for all $v>y$ we have $c(x,v)=c(x,y)$.\label{c(x_v)=c(x_y)}
		\end{enumerate}
	\end{lemma}
	\begin{proof}
		Assume that $\0$ is increasing on a maximal chain containing $u<x<y$, so  $\0(e_{uy})=[\0(e_{ux}),\0(e_{xy})]$. In view of \cref{vf(e_z)(u_u)-vf(e_z)(v_v),[vf(e_z)_0(e_xy)]=(vf(e_z)(u_u)-vf(e_z)(v_v))e_uv} we have
		\begin{align}\label{[0(e_uy)_vf(e_y)]=c(u_y)0(e_uy)}
			[\0(e_{uy}),\vf(e_y)]=c(u,y)\0(e_{uy}).
		\end{align}
		On the other hand, using the Jacobi identity we obtain
		\begin{align*}
			[\0(e_{uy}),\vf(e_y)]&=[[\0(e_{ux}),\0(e_{xy})],\vf(e_y)]\\
			&=-[[\0(e_{xy}),\vf(e_y)],\0(e_{ux})]-[[\vf(e_y),\0(e_{ux})],\0(e_{xy})].
		\end{align*}
		Observe that $[\0(e_{xy}),\vf(e_y)]=c(x,y)\0(e_{xy})$ by \cref{vf(e_z)(u_u)-vf(e_z)(v_v),[vf(e_z)_0(e_xy)]=(vf(e_z)(u_u)-vf(e_z)(v_v))e_uv} and $[\vf(e_y),\0(e_{ux})]=0$, because $[e_y,e_{ux}]=0$. Therefore,
		\begin{align}\label{[0(e_uy)_vf(e_y)]=c(x_y)0(e_uy)}
			[\0(e_{uy}),\vf(e_y)]=-c(x,y)[\0(e_{xy}),\0(e_{ux})]=c(x,y)\0(e_{uy}).
		\end{align}
		It follows from \cref{[0(e_uy)_vf(e_y)]=c(u_y)0(e_uy),[0(e_uy)_vf(e_y)]=c(x_y)0(e_uy)} that $c(u,y)=c(x,y)$, as desired. If $\0$ is decreasing on a maximal chain containing $u<x<y$, then  $\0(e_{uy})=-[\0(e_{ux}),\0(e_{xy})]$, and the rest of the proof is the same. Thus, \cref{c(u_y)=c(x_y)} is proved. The proof of \cref{c(x_v)=c(x_y)} is similar.
	\end{proof}
	
	
	\begin{proposition}\label{c=c}
		Let $x_1<\dots<x_m$ be a maximal chain in $X$. Then $c(x_i,x_j)=c(x_1,x_m)$ for all $1\le i<j\le m$. 
	\end{proposition}
	\begin{proof}
		Indeed, by \cref{c(x_y)=c(u_y)=c(x_v)} we have $c(x_i,x_j)=c(x_i,x_m)=c(x_1,x_m)$.
	\end{proof}
	
	The following example shows that the map $c$ may take different values on different maximal chains.
	
	\begin{example}
		Let $X=\{1,2,3,4,5\}$ with the following Hasse diagram.
		
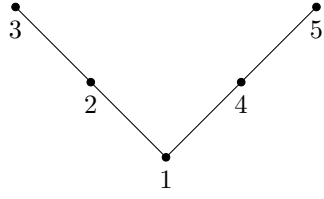
\begin{center}
			\begin{tikzpicture}
			\draw  (0,0)--(-1,1);
			\draw  (-1,1)--(-2,2);
			\draw  (0,0)-- (1,1);
			\draw  (1,1)-- (2,2);
			\draw [fill=black] (0,0) circle (0.05);
			\draw  (0,-0.3) node {$1$};
			\draw [fill=black] (-1,1) circle (0.05);
			\draw  (-1,0.7) node {$2$};
			\draw [fill=black] (-2,2) circle (0.05);
			\draw  (-2,1.7) node {$3$};
			\draw [fill=black] (1,1) circle (0.05);
			\draw  (1,0.7) node {$4$};
			\draw [fill=black] (2,2) circle (0.05);
			\draw  (2,1.7) node {$5$};
			\end{tikzpicture}
		\end{center}
		Consider the $K$-linear map $\vf:I(X,K)\to I(X,K)$ defined as follows: $\vf(e_1)=e_1+2e_4+2e_5$, $\vf(e_{12})=e_{12}$, $\vf(e_{13})=e_{13}$, $\vf(e_{14})=e_{14}$, $\vf(e_{15})=-e_{15}$, $\vf(e_2)=e_2$, $\vf(e_{23})=e_{23}$, $\vf(e_3)=e_3$, $\vf(e_4)=-e_4$, $\vf(e_{45})=e_{45}$, $\vf(e_5)=-e_5$. 
		
		Observe that $\vf(e_{xy})=\pm e_{xy}$ for all $x\le y$ except $x=y=1$, so \cref{[vf-comm-pres-on-the-basis]} holds, whenever $(x,y)\ne(1,1)$. Moreover, observe that $\vf(D(X,K))\sst D(X,K)$, so $[\vf(e_1),\vf(e_x)]=0$ for all $x\in X$. The rest of the commutators that should be calculated to justify \cref{[vf-comm-pres-on-the-basis],[vf(e_xz)_vf(e_zy)]=const} are:
		\begin{align*}
			[\vf(e_1),\vf(e_{23})]=[\vf(e_1),\vf(e_{45})]=0
		\end{align*}
		and
		\begin{align*}
			e_{12}&=[\vf(e_1),\vf(e_{12})]=[\vf(e_{12}),\vf(e_2)],\\
			e_{13}&=[\vf(e_1),\vf(e_{13})]=[\vf(e_{12}),\vf(e_{23})]=[\vf(e_{13}),\vf(e_3)],\\
			-e_{14}&=[\vf(e_1),\vf(e_{14})]=[\vf(e_{14}),\vf(e_4)],\\
			e_{15}&=[\vf(e_1),\vf(e_{15})]=[\vf(e_{14}),\vf(e_{45})]=[\vf(e_{15}),\vf(e_5)],\\
			e_{23}&=[\vf(e_2),\vf(e_{23})]=[\vf(e_{23}),\vf(e_3)],\\
			-e_{45}&=[\vf(e_4),\vf(e_{45})]=[\vf(e_{45}),\vf(e_5)].
		\end{align*}
		Thus, $\vf$ is a strong commutativity preserver by \cref{vf-on-the-natural-basis,[vf(e_x)_vf(e_xy)]-LI}. It is easily seen by \cref{ker-vf-strong} that $\vf$ is bijective.
		
		Now observe that $\0=\id_{\B}$, $\sg(1,2)=\sg(1,3)=\sg(1,4)=\sg(2,3)=\sg(4,5)=1$, $\sg(1,5)=-1$ and $c(1,2)=c(1,3)=c(2,3)=1$, $c(1,4)=c(1,5)=c(4,5)=-1$. This shows that $c$ takes distinct values on the maximal chains $1<2<3$ and $1<4<5$ of $X$.
	\end{example}
	
	\subsection{From $(\0,\sg,c)$ to $\vf$}
	
	In \cref{sec-from-(0_sg_nu_c)-to-vf} we studied bijective strong commutativity preservers $\vf$ of $I(X,K)$ which also preserve $D(X,K)$, and with any such map we associated the quadruple of simpler maps $(\0,\sg,\nu,c)$. We proved several properties of $\0$, $\sg$, $\nu$ and $c$. In particular, we observed that it is enough to consider $\nu=0$, which means that $\vf$ is pure, i.e. $\vf$ satisfies
	\begin{align}\label{vf(e_xy)=sg(x_y)0(e_xy)}
		\vf(e_{xy})=\sg(x,y)\0(e_{xy}),\ x<y.
	\end{align}
	We are going to show that the properties of the triple $(\0,\sg,c)$ proved in \cref{sec-from-(0_sg_nu_c)-to-vf} are also sufficient to determine a pure commutativity preserver $\vf:I(X,K)\to I(X,K)$.
	
	\begin{definition}\label{sigma_compativel}
		Let $\0:\B\to\B$ be a bijection and $\sigma:X^2_{<}\to K^*$ a map. If there is $c:X^2_{<}\to K^*$ satisfying \cref{defn-of-c(x_z)}, then we say that $\sigma$ is \emph{$c$-compatible} with $\0$.
	\end{definition}

	\begin{proposition}\label{vf-comm-pres-on-I(X_K)}
		Let $\varphi: I(X,K)\to I(X,K)$ be a $K$-linear map satisfying \cref{vf(e_xy)=sg(x_y)0(e_xy)} on $J(I(X,K))$, where $\0:\B\to\B$ is a bijection which is monotone on maximal chains, and $\sg:X^2_{<}\to K^*$ is $c$-compatible with $\0$ for some $c:X^2_{<}\to K^*$. If, moreover, $\vf(D(X,K))\sst D(X,K)$ and \cref{vf(e_z)(u_u)-vf(e_z)(v_v)} holds, then $\varphi$ is a strong commutativity preserver.
	\end{proposition}
	\begin{proof}
		Let $x\le y$ and $u\le v$ such that $[e_{xy},e_{uv}]=0$.
		
		\textit{Case 1.} $x<y$ and $u<v$. Let $\0(e_{xy})=e_{ab}$ and $\0(e_{uv})=e_{cd}$. If $b=c$ and $\0\m$ is increasing (resp.~decreasing) on a maximal chain containing $a<b<d$, then $x<y=u<v$ (resp.~$u<v=x<y$). If $a=d$, then the same argument applied to a maximal chain containing $c<d<b$ also results in $y=u$ or $x=v$. This contradicts $[e_{xy},e_{uv}]=0$. Thus, $b\ne c$ and $a\ne d$, so $[\0(e_{xy}),\0(e_{uv})]=0$. It immediately follows from \cref{vf(e_xy)=sg(x_y)0(e_xy)} that $[\vf(e_{xy}),\vf(e_{uv})]=0$.
		
		\textit{Case 2.} $x=y$ and $u<v$. Then $x\not\in\{u,v\}$, so $[\vf(e_{xy}),\0(e_{uv})]=0$ by \cref{vf(e_z)(u_u)-vf(e_z)(v_v)} and the fact that $\vf(e_{xy})\in D(X,K)$. Hence, $[\vf(e_{xy}),\vf(e_{uv})]=0$ in view of \cref{vf(e_xy)=sg(x_y)0(e_xy)}.
		
		\textit{Case 3.} $x<y$ and $u=v$. This case is symmetric to Case 2.
		
		\textit{Case 4.} $x=y$ and $u=v$. Then $[\vf(e_{xy}),\vf(e_{uv})]=0$ due to $\vf(D(X,K))\sst D(X,K)$.
		
		We have thus proved \cref{[vf-comm-pres-on-the-basis]}.
		
		Let $x<z<y$. If $\0$ is increasing on a maximal chain containing $x<z<y$, then  $\0(e_{xz})\0(e_{zy})=\0(e_{xy})$, so by \cref{vf(e_xy)=sg(x_y)0(e_xy)} and the $c$-compatibility of $\sg$,
		\begin{align*}
			[\vf(e_{xz}),\vf(e_{zy})]=\sg(x,z)\sg(z,y)\0(e_{xy})=c(x,y)\sg(x,y)\0(e_{xy}).
		\end{align*}
		On the other hand, by \cref{vf(e_z)(u_u)-vf(e_z)(v_v)} and the fact that $\vf(e_{xy})\in D(X,K)$,
		\begin{align*}
			[\vf(e_x),\vf(e_{xy})]=c(x,y)\sg(x,y)\0(e_{xy})=-[\vf(e_y),\vf(e_{xy})]=[\vf(e_{xy}),\vf(e_y)].
		\end{align*}
		The case when $\0$ is decreasing on a maximal chain containing $x<z<y$ (so that $\0(e_{zy})\0(e_{xz})=\0(e_{xy})$) is similar. Hence, \cref{[vf(e_xz)_vf(e_zy)]=const} holds too. By \cref{vf-on-the-natural-basis} the map $\vf$ is a commutativity preserver. Since $[\vf(e_x),\vf(e_{xy})]$ is a nonzero multiple of $\0(e_{xy})$ and $\0$ is a bijection, then $\{[\vf(e_x),\vf(e_{xy})]:x<y\}$ is linearly independent. So, $\vf$ is strong thanks to \cref{[vf(e_x)_vf(e_xy)]-LI}.
	\end{proof}
	
	The following lemma is an analogue of \cite[Lemma 5.8]{FKS} and its proof is also very similar. It shows that, under certain hypotheses, the values of a commutativity preserver $\vf$ on $D(X,K)$ depend only on $\{\vf(e_z)(u_0,u_0)\}_{z\in X}$, where $u_0$ is a fixed element of $X$.
	
	\begin{lemma}\label{0_vf=0_psi-and-vf(e_z)_uu=psi(e_z)_uu}
		Let $z,u_0\in X$ and $\vf,\psi:I(X,K)\to I(X,K)$ be bijective strong commutativity preservers which preserve diagonality. If $\0_\vf=\0_\psi$, $c_\vf=c_\psi$ and $\vf(e_z)(u_0,u_0)=\psi(e_z)(u_0,u_0)$, then $\vf(e_z)=\psi(e_z)$.
	\end{lemma}
	\begin{proof}
		Given $v\in X$, we choose a walk $\G:u_0,u_1,\dots,u_m=v$ and $v_i<w_i$, $0\le i\le m-1$, such that
		\begin{align}\label{vf(e_v_iw_i)=e_u_iu_i+1}
			\0_{\vf}(e_{v_iw_i})=
			\begin{cases}
				e_{u_iu_{i+1}}, & u_i<u_{i+1},\\
				e_{u_{i+1}u_i}, & u_i>u_{i+1}.
			\end{cases}
		\end{align}
		Then
		\begin{align}\label{vf(e_u)_vv=vf(e_u)_u_0u_0+sum}
			\vf(e_z)(v,v)=\vf(e_z)(u_0,u_0)+\sum_{i=0}^{m-1}\Delta_{\G,i}(z),
		\end{align}
		where $\Delta_{\G,i}(z)=\vf(e_z)(u_{i+1},u_{i+1})-\vf(e_z)(u_i,u_i)$. By \cref{vf(e_z)(u_u)-vf(e_z)(v_v)}
		\begin{align}\label{vf(e_z)_u_(i+1)u_(i+1)-vf(e_z)_u_iu_i}
			\Delta_{\G,i}(z) & =
			\begin{cases}
				-c_\vf(v_i,w_i), & (u_i<u_{i+1}\wedge z=v_i)\vee(u_i>u_{i+1}\wedge z=w_i),\\
				c_\vf(v_i,w_i),& (u_i<u_{i+1}\wedge z=w_i)\vee(u_i>u_{i+1}\wedge z=v_i),\\
				0, & z\not\in\{v_i,w_i\}.
			\end{cases}
		\end{align}
		In view of \cref{vf(e_z)_u_(i+1)u_(i+1)-vf(e_z)_u_iu_i,vf(e_v_iw_i)=e_u_iu_i+1}, the right-hand side of \cref{vf(e_u)_vv=vf(e_u)_u_0u_0+sum} depends only on $\0_\vf$, $c_\vf$ and $\vf(e_z)(u_0,u_0)$, so $\vf(e_z)(v,v)=\psi(e_z)(v,v)$. Thus,  $\vf(e_z)=\psi(e_z)$.
	\end{proof}
	
	\begin{definition}
		Let $\0:\B\to \B$ be a bijection, $c:X^2_{<}\to K^*$ and $\G: u_0,u_1,\dots,u_m=u_0$ a closed walk in $X$. Define the following $4$ functions $X\to K$:
		\begin{align*}
			s^+_{\0,c,\G}(z)&=\sum c(z,w_i), \text{ where the sum is over those } 0\le i\le m-1\text{ such that }\\
			&\quad u_i<u_{i+1}\text{ and }\exists w_i>z\text{ with }\0(e_{zw_i})=e_{u_iu_{i+1}},\\
			s^-_{\0,c,\G}(z)&=\sum c(z,w_i), \text{ where the sum is over those } 0\le i\le m-1\text{ such that }\\
			&\quad u_i>u_{i+1}\text{ and }\exists w_i>z\text{ with }\0(e_{zw_i})=e_{u_{i+1}u_i},\\
			t^+_{\0,c,\G}(z)&=\sum c(z,w_i), \text{ where the sum is over those } 0\le i\le m-1\text{ such that }\\
			&\quad u_i<u_{i+1}\text{ and }\exists w_i<z\text{ with }\0(e_{w_iz})=e_{u_iu_{i+1}},\\
			t^-_{\0,c,\G}(z)&=\sum c(z,w_i), \text{ where the sum is over those } 0\le i\le m-1\text{ such that }\\
			&\quad u_i>u_{i+1}\text{ and }\exists w_i<z\text{ with }\0(e_{w_iz})=e_{u_{i+1}u_i}.
		\end{align*}
	\end{definition}
	
	As in~\cite[Remark 5.10]{FKS} we obtain the following.
	\begin{remark}\label{0_vf-is-admissible}
		Let $\vf:I(X,K)\to I(X,K)$ be a bijective strong commutativity and diagonality preserver, $\0=\0_\vf$ and $c=c_\vf$. Then for any closed walk $\Gamma:u_0,u_1,\dots,u_m=u_0$ in $X$ and for all $z\in X$
		\begin{align}\label{s^+-s^-=t^+-t^-2}
			s^+_{\0,c,\G}(z)-s^-_{\0,c,\G}(z)=t^+_{\0,c,\G}(z)-t^-_{\0,c,\G}(z).
		\end{align}
	\end{remark}
	
	\begin{definition}\label{defn-admissible}
		Let $\0:\B\to \B$ be a bijection and $c:X^2_{<}\to K^*$ a map. The pair $(\0,c)$ is said to be \textit{admissible} if \cref{s^+-s^-=t^+-t^-2} holds for any closed walk $\G:u_0,u_1,\dots,u_m=u_0$ in $X$ and for all $z\in X$. 
	\end{definition}
	
	We remark that even if $c:X^2_{<}\to K^*$ is a constant map, there exist examples of non-admissible pairs $(\0,c)$ (see, e.g., \cite[Example 5.12]{FKS}).
	
	\begin{remark}\label{sum-vf(e_z)(uu)-nonzero}
		Let $\vf:I(X,K)\to I(X,K)$ be a bijective commutativity preserver and $u,v\in X$. Then $\sum_{z\in X}\vf(e_z)(u,u)=\sum_{z\in X}\vf(e_z)(v,v)\ne 0$.
		
		Indeed, since $\dl$ is central and $\vf$ is bijective, then $\vf(\delta)$ is a non-zero central element of $I(X,K)$.
	\end{remark}
	
	\begin{definition}
		Let $\0:\B\to \B$ be a bijection and $\sg:X^2_{<}\to K^*$ a map. Let $\vf:I(X,K)\to I(X,K)$ be a bijective strong commutativity and diagonality preserver. We say that $\vf$ \textit{induces} the pair $(\0,\sg)$ if \cref{vf(e_xy)=sg(x_y)0(e_xy)} holds for all $x<y$.
	\end{definition}
	
	\begin{definition}\label{c_constante}
		A map $c:X^2_{<}\to K^*$ is said to be \emph{constant on maximal chains in $X$} if \cref{c=c} holds.  
	\end{definition}
	
	The proof of the following lemma is similar to that of \cite[Lemma 5.16]{FKS}, but we give it here for completeness.
	
	\begin{lemma}\label{existense}
		Let $\0:\B\to \B$ be a bijection which is monotone on maximal chains in $X$ and $\sg:X^2_{<}\to K^*$ $c$-compatible with $\0$ for some $c:X^2_{<}\to K^*$, such that $c$ is constant on maximal chains in $X$ and the pair $(\0,c)$ is admissible. Then there exists a pure commutativity preserver $\vf:I(X,K)\to I(X,K)$ that induces $(\0,\sg)$.	
	\end{lemma}
	\begin{proof}
		We define $\vf$ on $J(I(X,K))$ by \cref{vf(e_xy)=sg(x_y)0(e_xy)}. Now, let us show that $\vf$ extends to an injective strong commutativity preserver $I(X,K)\to I(X,K)$. Fix $u_0\in X$ and $\vf(e_z)(u_0,u_0)\in K$ ($z\in X$) in a way that $\sum_{z\in X}\vf(e_z)(u_0,u_0)\ne 0$. Given $v\in X$, there is a walk $\G: u_0,u_1,\dots,u_m=v$ from $u_0$ to $v$, since $X$ is connected. Then we define $\vf(e_z)(v,v)$ by formulas \cref{vf(e_u)_vv=vf(e_u)_u_0u_0+sum,vf(e_z)_u_(i+1)u_(i+1)-vf(e_z)_u_iu_i}. We will show that the definition does not depend on the choice of a walk from $u_0$ to $v$. For, if there is another walk $\G': u'_0=u_0,u'_1,\dots,u'_{m'}=v$ from $u_0$ to $v$, then there is a closed walk $\Omega: u_0,u_1,\dots,u_m=v=u'_{m'},u_{m+1}=u'_{m'-1},\dots,u_{m+m'}=u'_0=u_0$. Since $(\0,c)$ is admissible,
		\begin{align*}
			0&=-s^+_{\0,c,\Omega}(z)+s^-_{\0,c,\Omega}(z)+t^+_{\0,c,\Omega}(z)-t^-_{\0,c,\Omega}(z)\\
			&=\sum_{i=0}^{m+m'-1}\Delta_{\Omega,i}(z)=\sum_{i=0}^{m-1}\Delta_{\G,i}(z)+\sum_{i=m}^{m+m'-1}\Delta_{\Omega,i}(z),
		\end{align*}
		where $\Delta_{\Omega,i}(z)$ and $\Delta_{\G,i}(z)$ are given by \cref{vf(e_z)_u_(i+1)u_(i+1)-vf(e_z)_u_iu_i}. Hence
		\begin{align*}
			\sum_{i=0}^{m-1}\Delta_{\G,i}(z)&=-\sum_{i=m}^{m+m'-1}\Delta_{\Omega,i}(z)=\sum_{i=0}^{m'-1}(-\Delta_{\Omega,i+m}(z))\\
			&=\sum_{i=0}^{m'-1}\Delta_{\G',m'-i-1}(z)=\sum_{i=0}^{m'-1}\Delta_{\G',i}(z).
		\end{align*}
		So, $\vf(e_z)(v,v)$ is well-defined by \cref{vf(e_u)_vv=vf(e_u)_u_0u_0+sum}. Put $\vf(e_z)=\sum_{v\in X}\vf(e_z)(v,v)e_v\in D(X,K)$. Automatically, $\vf(D(X,K))\subseteq D(X,K)$.
		
		We are going to prove \cref{vf(e_z)(u_u)-vf(e_z)(v_v)}. Let $e_{xy}\in\B$ with $\0(e_{xy})=e_{uv}$. The first equality of \cref{vf(e_z)(u_u)-vf(e_z)(v_v)} is due to the fact that $\vf(D(X,K))\subseteq D(X,K)$. To prove the second one, take a walk starting at $u_0$ of the form $\G: u_0,\dots,u_l=u,\dots,u_m=v$, where $u_l<\dots<u_m$. Using this walk to define $\vf(e_z)(u,u)$ and $\vf(e_z)(v,v)$, by \cref{vf(e_u)_vv=vf(e_u)_u_0u_0+sum} we have
		\begin{align}\label{vf(e_z)(v_v)-vf(e_z)(u_u)=sum}
			\vf(e_z)(v,v)-\vf(e_z)(u,u)=\sum_{i=l}^{m-1}\Delta_{\G,i}(z).
		\end{align}
		Assume that $\0\m$ is increasing on a maximal chain containing $u_{l}<\dots<u_m$. Then there exist $v_{l}<\dots<v_m$ such that $\0(e_{v_iv_j})=e_{u_iu_j}$, $l\le i<j\le m$. If $z=v_{l}$, then $\Delta_{\G,l}(z)=-c(v_l,v_{l+1})=-c(v_l,v_m)$ and $\Delta_{\G,i}(z)=0$ for all $l<i\le m-1$ by \cref{vf(e_z)_u_(i+1)u_(i+1)-vf(e_z)_u_iu_i}. Hence, the right-hand side of \cref{vf(e_z)(v_v)-vf(e_z)(u_u)=sum} equals $-c(v_l,v_m)$. But $\0(e_{zv_m})=e_{u_lu_m}=e_{uv}$, so \cref{vf(e_z)(u_u)-vf(e_z)(v_v)} holds for $u<v$. Similarly, if $z=v_m$, then $\Delta_{\G,m-1}(z)=c(v_{m-1},v_m)=c(v_l,v_m)$ and $\Delta_{\G,i}(z)=0$ for all $l\le i<m-1$ by \cref{vf(e_z)_u_(i+1)u_(i+1)-vf(e_z)_u_iu_i}. Hence, the right-hand side of \cref{vf(e_z)(v_v)-vf(e_z)(u_u)=sum} equals $c(v_l,v_m)$ and $\0(e_{v_lz})=e_{uv}$ justifying \cref{vf(e_z)(u_u)-vf(e_z)(v_v)} for $u<v$. If $z=v_k$ with $l<k<m$, then $\Delta_{\G,k}(z)=-c(v_{k},v_{k+1})=-c(v_l,v_m)$, $\Delta_{\G,k-1}(z)=c(v_{k-1},v_{k})=c(v_l,v_m)$ and $\Delta_{\G,i}(z)=0$ for all $l\le i\le m-1$, $i\not\in\{k-1, k\}$. Hence, the right-hand side of \cref{vf(e_z)(v_v)-vf(e_z)(u_u)=sum} equals $0$ and $\0(e_{v_lv_m})=e_{uv}$ with $z\not\in\{v_l,v_m\}$, so \cref{vf(e_z)(u_u)-vf(e_z)(v_v)} still holds. Finally, if $z\ne v_i$ for all $l\le i\le m$, then $\Delta_{\G,i}(z)=0$ for all $l\le i\le m-1$, so that the right-hand side of \cref{vf(e_z)(v_v)-vf(e_z)(u_u)=sum} again equals $0$. The decreasing case is similar.
		
		By \cref{vf-comm-pres-on-I(X_K)}, the extended $\vf$ is a strong commutativity preserver. Clearly, $\0_\vf=\0$ and $\sg_\vf=\sg$ by the construction.
		
		It remains to prove that $\vf$ is injective (and hence, bijective). But this follows from our choice of $\vf(e_z)(u_0,u_0)\in K$ for all $z\in X$ and \cref{ker-vf-strong}, because $0\ne \sum_{z\in X}\vf(e_z)(u_0,u_0)=\vf(\dl)(u_0,u_0)$.
	\end{proof}
	
	\begin{definition}\label{defn-of-tau_0_sg_nu_s}
		Let $X=\{x_1,\dots,x_n\}$. Consider a bijection $\0:\B\to \B$ which is monotone on maximal chains in $X$, $\sg:X^2_{<}\to K^*$ $c$-compatible with $\0$ for some $c:X^2_{<}\to K^*$, such that $c$ is constant on maximal chains in $X$ and the pair $(\0,c)$ is admissible, and a sequence $\kappa=(k_1,\dots,k_n)\in K^n$ such that $\sum_{i=1}^nk_i\in K^*$. We denote by $\tau=\tau_{\0,\sg,c,\kappa}$ the pure commutativity preserver of $I(X,K)$ such that $\tau|_{J(I(X,K))}$ is given by the right-hand side of \cref{vf(e_xy)=sg(x_y)0(e_xy)} and $\tau|_{D(X,K)}$ is determined by 
		\begin{align}\label{tau(e_x_i)(x_1_x_1)=s_i}
			\tau(e_{x_i})(x_1,x_1)=k_i,
		\end{align}
		$i=1,\dots,n$, as in \cref{0_vf=0_psi-and-vf(e_z)_uu=psi(e_z)_uu,existense}.
	\end{definition}
	
	\begin{theorem}\label{vf-decomp-as-tau_0_sg_nu_s}
		Each pure commutativity preserver $\vf:I(X,K)\to I(X,K)$ can be uniquely represented in the form
		\begin{align}\label{vf=tau_0_sg_nu_s}
			\vf=\tau_{\0,\sg,c,\kappa},
		\end{align}
		where $\0:\B\to \B$ is a bijection which is monotone on maximal chains in $X$, $\sg:X^2_<\to K^*$ is $c$-compatible with $\0$ for some $c:X^2_{<}\to K^*$, such that $c$ is constant on maximal chains in $X$ and the pair $(\0,c)$ is admissible, and $\kappa=(k_1,\dots,k_n)\in K^n$ is such that $\sum_{i=1}^nk_i\in K^*$.
	\end{theorem}
	\begin{proof}
		Let $X=\{x_1,\dots,x_n \}$. We define $\0=\0_\vf$, $\sg=\sg_\vf$, $c=c_\vf$ and $k_i=\vf(e_{x_i})(x_1,x_1)$, $i=1,\dots,n$. Then $\0$, $\sg$, $c$ and $\kappa$ satisfy the desired properties by \cref{0_monotone,defn-of-c(x_z),c=c,0_vf-is-admissible,sum-vf(e_z)(uu)-nonzero}.
		Hence, $\tau:=\tau_{\0,\sg,c,\kappa}$ is well-defined. For all $x_i<x_j$ we have
		$\vf(e_{x_ix_j})=\sg(x_i,x_j)\0(e_{x_ix_j})=\tau(e_{x_ix_j}),$
		so $\vf|_{J(I(X,K))}=\tau|_{J(I(X,K))}$. Moreover, $\vf|_{D(X,K)}=\tau|_{D(X,K)}$ by  \cref{0_vf=0_psi-and-vf(e_z)_uu=psi(e_z)_uu}. The uniqueness of \cref{vf=tau_0_sg_nu_s} follows from \cref{vf(e_xy)=sg(x_y)0(e_xy),tau(e_x_i)(x_1_x_1)=s_i}.
	\end{proof}
	
	As a consequence of \cref{S_af-circ-vf-pure,vf-decomp-as-tau_0_sg_nu_s}, we obtain the main result of this paper.
	\begin{corollary}\label{vf-decomp-S_af-tau}
		Every bijective strong commutativity and diagonality preserver $I(X,K)\to I(X,K)$ is a composition of a commutativity preserver $S_\af$ of shift type and a pure commutativity preserver $\tau_{\0,\sg,c,\kappa}$.
	\end{corollary}
	
	\subsection{Applications and a problem for future research}
	
	\begin{definition}\label{Lie-type-def}
		We say that a strong commutativity preserver $A\to A$ is \textit{of Lie type} if it is the sum of a (non-zero) scalar multiple of a Lie automorphism of $A$ and a linear central-valued map on $A$.
	\end{definition}
	
	\begin{proposition}\label{vf-Lie-type}
		A bijective strong commutativity and diagonality preserver $\vf:I(X,K)\to I(X,K)$ is of Lie type if and only if $\vf=S_\af\circ\tau_{\0,\sg,c,\kappa}$, where $\af$ is central-valued and $c$ is a constant map on $X^2_<$.
	\end{proposition}
	\begin{proof}
		Let $\vf$ be a bijective strong commutativity and diagonality preserver of $I(X,K)$. By \cref{vf-decomp-S_af-tau} we have $\vf=S_\af\circ\tau_{\0,\sg,c,\kappa}$, where $\af$ is given by \cref{af(0(e_xy))=nu(e_xy),af_D(X_K)=0}. 
		
		Assume that $\vf=k\psi+\xi$, where $k\in K^*$, $\psi$ is a Lie automorphism of $I(X,K)$ and $\xi:I(X,K)\to Z(I(X,K))$. Then $\xi(e_{xy})=\nu(e_{xy})$ and $k\psi(e_{xy})=\sg(x,y)\0(e_{xy})$ for all $x<y$ by \cref{vf(e_xy)=0(e_xy)+nu(e_xy)} and \cite[Propositions 2.2--2.3]{FKS}. Hence $\af(\0(e_{xy}))\in Z(I(X,K))$ for all $x<y$ by \cref{af(0(e_xy))=nu(e_xy)}, so together with \cref{af_D(X_K)=0} this guarantees that $\af$ is central-valued. Now, it follows from $[e_x,e_{xy}]=e_{xy}$ and \cref{vf(e_z)(u_u)-vf(e_z)(v_v)} that
		\begin{align*}
			k\sg(x,y)\0(e_{xy})&=k^2\psi(e_{xy})=[k\psi(e_x),k\psi(e_{xy})]=[k\psi(e_x)+\xi(e_x),k\psi(e_{xy})]\\
			&=[\vf(e_x),\sg(x,y)\0(e_{xy})]=c(x,y)\sg(x,y)\0(e_{xy}),
		\end{align*}
		whence $c(x,y)=k$ for all $x<y$.
		
		Conversely, suppose that $\af(e_{xy})\in Z(I(X,K))$ for all $x\le y$ and $k:=c(x,y)\in K^*$ for all $x<y$. Define $\psi:I(X,K)\to I(X,K)$ by
		\begin{align*}
			\psi(e_{xy})=
			\begin{cases}
				k\m\sg(x,y)\0(e_{xy}), & x<y,\\
				k\m\vf(e_{xy}), & x=y.
			\end{cases}
		\end{align*}
		We are going to prove that $\psi$ preserves the Lie product. Take arbitrary $x\le y$ and $u\le v$.
		
		\textit{Case 1.} $x<y$ and $u<v$. If $[e_{xy},e_{uv}]=0$, then $[\psi(e_{xy}),\psi(e_{uv})]=0$ by \cref{0-strong-comm-pres}. If $y=u$, then $[e_{xy},e_{uv}]=e_{xv}$ and by \cref{[0(e_xy)_0(e_yz)]=+-0(e_xz),defn-of-c(x_z)}
		\begin{align*}
			[\psi(e_{xy}),\psi(e_{uv})]=k^{-2}\sg(x,y)\sg(y,v)[\0(e_{xy}),\0(e_{yv})]=k\m\sg(x,v)\0(e_{xv})=\psi(e_{xv}).
		\end{align*}
		The case $x=v$ reduces to the case $y=u$ because $[e_{xy},e_{ux}]=-[e_{ux},e_{xy}]$. 
		
		\textit{Case 2.} $x=y$ and $u<v$. Again, if $[e_x,e_{uv}]=0$, then by \cref{vf(e_xy)=0(e_xy)+nu(e_xy)} and $\vf(D(X,K))=D(X,K)$
		\begin{align*}
			[\psi(e_x),\psi(e_{uv})]=[k\m\vf(e_x),k\m\sg(u,v)\0(e_{uv})]=k^{-2}[\vf(e_x),\vf(e_{uv})]=0.
		\end{align*}
		And if $[e_x,e_{uv}]\ne 0$, then $u=x$ or $v=x$. For $u=x$ we have $[e_x,e_{uv}]=e_{xv}$ and by \cref{vf(e_z)(u_u)-vf(e_z)(v_v)}
		\begin{align*}
			[\psi(e_x),\psi(e_{uv})]=k^{-2}\sg(x,v)[\vf(e_x),\0(e_{xv})]=k\m\sg(x,v)\0(e_{xv})=\psi(e_{xv}).
		\end{align*}
		For $v=x$ we have $[e_x,e_{uv}]=-e_{ux}$ and
		\begin{align*}
			[\psi(e_x),\psi(e_{uv})]=k^{-2}\sg(u,x)[\vf(e_x),\0(e_{ux})]=-k\m\sg(u,x)\0(e_{ux})=-\psi(e_{ux})
		\end{align*}
		thanks to the same \cref{vf(e_z)(u_u)-vf(e_z)(v_v)}. 
		
		\textit{Case 3.} $x<y$ and $u=v$. This case is similar to Case 2.
		
		\textit{Case 4.} $x=y$ and $u=v$. Then $[\psi(e_{xy}),\psi(e_{uv})]=[\vf(e_{xy}),\vf(e_{uv})]=0$ due to $\vf(D(X,K))=D(X,K)$.
		
		Thus, $\psi$ preserves the Lie product. It is also a bijection of $I(X,K)$, because $\0$ is a bijection of $\B$ and $\vf$ is bijective on $D(X,K)$. So, $\psi$ is a Lie automorphism of $I(X,K)$. 
		
		Now define $\xi(e_{xy})=\sg(x,y)\af(\0(e_{xy}))$ for all $x<y$ and $\xi(e_x)=0$ for all $x\in X$, so that $\xi$ is central-valued. If $x<y$, then by \cref{vf(e_xy)=0(e_xy)+nu(e_xy),af(0(e_xy))=nu(e_xy)}
		\begin{align*}
			\vf(e_{xy})=\sg(x,y)\0(e_{xy})+\nu(e_{xy})=k\psi(e_{xy})+\sg(x,y)\af(\0(e_{xy}))=k\psi(e_{xy})+\xi(e_{xy}),
		\end{align*}
		and 
		\begin{align*}
			\vf(e_x)=k\psi(e_x)=k\psi(e_x)+\xi(e_x).
		\end{align*}
	\end{proof}
	
	Recall that $T_n(K)\cong I(X,K)$, where $X$ is a chain of length $n-1$. As a consequence of \cref{vf-Lie-type} we obtain a partial generalization of \cite[Theorem~4]{Marcoux-Sourour99}.
	\begin{corollary}
		Let $n>2$ and $X$ be a chain of length $n-1$. Then every bijective strong commutativity and diagonality preserver $I(X,K)\to I(X,K)$ is of Lie type.
	\end{corollary}
	\begin{proof}
		We know that $\vf=S_\af\circ\tau_{\0,\sg,c,\kappa}$ thanks to \cref{vf-decomp-S_af-tau}, where $\af$ is as in \cref{af(0(e_xy))=nu(e_xy),af_D(X_K)=0}. In view of \cref{vf-Lie-type} it suffices to prove that $\af$ is central-valued and $c$ is constant. The former is due to \cref{nu-almost-central-valued}, and the latter is because of the fact that $X$ has only one maximal chain.
	\end{proof}
	
	\begin{example}
		Let $X=\{1,2\}$ be a chain of length $1$, where $1<2$. Then $\vf:I(X,K)\to I(X,K)$, given by $\vf(e_1)=e_1$, $\vf(e_{12})=e_1+e_{12}$ and $\vf(e_2)=e_2$, is a strong bijective commutativity and diagonality preserver of $I(X,K)$ by \cref{vf-on-the-natural-basis,[vf(e_x)_vf(e_xy)]-LI}, since $[\vf(e_1),\vf(e_2)]=0$ and $[\vf(e_1),\vf(e_{12})]=e_{12}=[\vf(e_{12}),\vf(e_2)]$.\footnote{Observe that $\vf=S_\af$, where $\af(e_1)=\af(e_2)=0$ and $\vf(e_{12})=e_1$.} If $\vf$ were of Lie type, say $\vf=k\psi+\xi$, then $\vf(e_{12})$ would be the sum of $k\psi(e_{12})\in J(I(X,K))$ by \cite[Propositions 2.2--2.3]{FKS} and $\xi(e_{12})\in Z(I(X,K))$, so $\vf(e_{12})_D=\xi(e_{12})\in Z(I(X,K))$, which is not the case.
	\end{example}
	
	Recall from~\cite[Theorem 4.15]{FKS} that any Lie automorphism of $I(X,K)$ is a composition of an inner automorphism of $I(X,K)$ and a Lie automorphism of $I(X,K)$ that preserves the diagonal subalgebra $D(X,K)$. It is natural to expect that an analogous fact holds for strong commutativity preservers of $I(X,K)$. Unfortunately, this is not the case as the following easy example shows.
	
	\begin{example}\label{exm-vf(e_x)_D=0}
		Let $X=\{1,2\}$ with $1<2$. Then $\vf(e_1)=e_{12}$, $\vf(e_{12})=e_2$ and $\vf(e_2)=\dl-e_{12}$ define a strong bijective commutativity preserver of $I(X,K)$, because $[\vf(e_1),\vf(e_2)]=0$ and $[\vf(e_1),\vf(e_{12})]=e_{12}=[\vf(e_{12}),\vf(e_2)]$ (see \cref{vf-on-the-natural-basis,[vf(e_x)_vf(e_xy)]-LI}). However, $\vf(e_1)$ is not conjugated to a diagonal element, since $\vf(e_1)_D=0$.
	\end{example}
	
	Nevertheless, the problem that $\vf(e_1)_D=0$ from \cref{exm-vf(e_x)_D=0} can be fixed by composing $\vf$ with $S_\af$ for an appropriate $\af:I(X,K)\to D(X,K)$. Indeed, define $\af(e_1)=\af(e_{12})=-\af(e_2)=e_1$. Then $S_\af$ is a bijective strong commutativity preserver by \cref{S_af-strong-comm-pres}. Moreover, $(S_\af\circ\vf)(e_1)=e_1+e_{12}$ and $(S_\af\circ\vf)(e_2)=e_2-e_{12}$ are commuting (in fact, orthogonal) idempotents, so there exists an inner automorphism $\psi$ of $I(X,K)$ such that $(\psi\circ S_\af\circ\vf)(e_1)=e_1$ and $(\psi\circ S_\af\circ\vf)(e_2)=e_2$ by \cite[Lemma 5.4]{GK}. Thus, the composition $\psi\circ S_\af\circ\vf$ is a bijective strong commutativity and diagonality preserver of $I(X,K)$. We conjecture that this can be done in a more general case.
	
	\begin{conjecture}
		Let $X$ be a finite connected poset and $\vf:I(X,K)\to I(X,K)$ a bijective strong commutativity preserver. Then there are a bijective strong commutativity preserver of shift type $S_\af$ and an inner automorphism $\psi$ of $I(X,K)$ such that $\psi\circ S_\af\circ\vf$ preserves $D(X,K)$.
	\end{conjecture}
	\section*{Acknowledgements}
	The second author was partially supported by CNPq (process: 404649/2018-1).
	
	\bibliography{bibl}{}
	\bibliographystyle{acm}
	
\end{document}